\documentclass[11pt,a4paper,reqno]{amsart}
\pdfoutput=1
\usepackage{tikz,enumitem,rotating,latexsym,bm,stmaryrd}
\usetikzlibrary{positioning,intersections}
\usepackage{amsmath,amsthm,amsfonts,amssymb,mathrsfs,pb-diagram}
\usepackage[pagebackref=false,bookmarks=true,colorlinks=true,linktoc=page,citecolor=darkgreen,linkcolor=blue,urlcolor=mediumblue]{hyperref}
\usepackage{caption}
\usepackage{lipsum,wasysym}
\usepackage{mathtools}
\usepackage{cleveref,fullpage}
\usepackage{tikz,enumitem,rotating,latexsym,bm,stmaryrd}
\usepackage{mathdots}
\usetikzlibrary{positioning,intersections,shapes.geometric,calc,decorations.pathreplacing}
\usepackage{multicol}
\usepackage[all]{xy}
\usepackage{comment}

\usepackage{xcolor}
\definecolor{mediumblue}{rgb}{0.0, 0.0, 0.8}
\colorlet{darkgreen}{green!50!black}

\renewcommand{\trianglerighteq}{\trianglerighteqslant}
\renewcommand{\trianglelefteq}{\trianglelefteqslant}

\tikzset{wei/.style=
{red,double=red,double
distance=1pt}}

\usepackage{ltxtable}
\usepackage{tabularx}

\tikzset{wei2/.style={red,double=red,double
distance=1pt}}

\allowdisplaybreaks
\numberwithin{equation}{section}
\usepackage{scalefnt}

\newtheorem{thm}{Theorem}[section]

\newtheorem{lem}[thm]{Lemma}
\newtheorem{prop}[thm]{Proposition}
\newtheorem*{prop*}{Proposition}
\newtheorem*{thmA*}{Theorem A}
\newtheorem*{thmB*}{Theorem B}
\newtheorem*{thmC*}{Theorem C}

\newtheorem*{cor*}{Corollary}

\newtheorem*{conj*}{Semisimplicity  Conjecture}

\theoremstyle{remark}

\newtheorem{rmk}[thm]{Remark}  
\newtheorem*{notation}{Notation}

\newtheorem*{ackn}{Acknowledgements}
\newtheorem*{conv}{Convention}

\theoremstyle{definition}
\newtheorem{defin}[thm]{Definition}
\newtheorem{eg}[thm]{Example}

\newcommand{\rad}{\mathrm{rad}}
\newcommand{\res}{\mathrm{res}}

\newcommand{\Std}{{\sf Std}}

\newcommand{\Shape}{\operatorname{Shape}}

\newcommand{\Bip}{\mathsf{Bip}}
\newcommand{\Path}{{\rm Path}}

\newcommand{\la}{\lambda}
\newcommand{\lapr}{\lambda^{\prime}}

 \renewcommand{\SS}{\mathsf{S}}

    \newcommand{\lxr}{\longrightarrow}

  \newcommand{\ct}{\mathsf{ct}}  

\newcommand{\TT}{\mathsf{T}}
\newcommand{\ui}{\underline{i}}
\newcommand{\tT}{\mathsf{t}}
\newcommand{\blob}{\mathsf{B}_{d}^{\kappa}}
\renewcommand{\epsilon}{\varepsilon}
\newcommand{\Image}{\mathrm{Im}}
\newcommand{\sS}{\mathsf{s}}

\DeclareMathOperator{\Hom}{Hom}

\let\<=\langle
\let\>=\rangle

 \newcommand\releven{11}
  
 \newcommand\rten{10} \newcommand\rtwelve{12}

\tikzset{
    ultra thin/.style= {line width=0.05pt},
    very thin/.style=  {line width=0.2pt},
    thin/.style=       {line width=0.1pt},
    semithick/.style=  {line width=0.6pt},
    thick/.style=      {line width=0.8pt},
    very thick/.style= {line width=1.2pt},
    ultra thick/.style={line width=1.6pt}
}

\crefname{defn}{Definition}{Definitions}
\crefname{thm}{Theorem}{Theorems}
\crefname{prop}{Proposition}{Propositions}
\crefname{lemma}{Lemma}{Lemmas}
\crefname{cor}{Corollary}{Corollaries}
\crefname{conj}{Conjecture}{Conjectures}
\crefname{section}{Section}{Sections}
\crefname{subsection}{Subsection}{Subsections}
\crefname{eg}{Example}{Examples}
\crefname{figure}{Figure}{Figures}
\crefname{remark}{Remark}{Remarks}
\crefname{remark}{Remark}{Remarks}
\crefname{equation}{equation}{equation}

\Crefname{defn}{Definition}{Definitions}
\Crefname{thm}{Theorem}{Theorems}
\Crefname{prop}{Proposition}{Propositions}
\Crefname{lemma}{Lemma}{Lemmas}
\Crefname{cor}{Corollary}{Corollaries}
\Crefname{conj}{Conjecture}{Conjectures}
\Crefname{section}{Section}{Sections}
\Crefname{subsection}{Subsection}{Subsections}
\Crefname{eg}{Example}{Examples}
\Crefname{figure}{Figure}{Figures}
\Crefname{remark}{Remark}{Remarks}
\Crefname{remark}{Remark}{Remarks}

 \newlength{\mylen}
\setbox1=\hbox{$\bullet$}\setbox2=\hbox{\tiny$\bullet$}
\usepackage{keyval}

\newcommand\wht{\Yfillcolour{white}}
\newcommand\blue{\Yfillcolour{blue!30}}

\newcommand{\ten}{10}

\usepackage{amsmath}
\newcommand\Item[1][]{%
  \ifx\relax#1\relax  \item \else \item[#1] \fi
  \abovedisplayskip=0pt\abovedisplayshortskip=0pt~\vspace*{-\baselineskip}}

\def\Item{\item\abovedisplayskip=0pt\abovedisplayshortskip=5pt~\vspace*{-\baselineskip}}
\newcommand\blfootnote[1]{%
  \begingroup
  \renewcommand\thefootnote{}\footnote{#1}%
  \addtocounter{footnote}{-1}%
  \endgroup
}
\usepackage{genyoungtabtikz}

\begin{document}

\title[Bases of simple modules]{Bases and BGG resolutions  of simple modules \\ of Temperley--Lieb algebras of type B}
\author[D. Michailidis]{Dimitris Michailidis}
\address{School of Mathematics, Statistics and Actuarial Science, University of Kent, CT2 7FS, UK.}
\email{\texttt{dm563@kent.ac.uk}}

\maketitle

\begin{abstract}
We construct explicit bases of simple modules and Bernstein--Gelfand--Gelfand (BGG) resolutions of all simple modules of the (graded) Temperley--Lieb algebra of type B over a field of characteristic zero. 
\end{abstract}

\section*{Introduction}

Inspired by the study of certain models in physics, Martin and Saleur introduced the main hero of this paper, the \textit{Temperley-Lieb algebra of type B} or \textit{blob algebra}, as the diagrammatic two parameter generalisation of the Temperley-Lieb algebra of type A \cite{MaSa}. The blob algebra deals with boundary conditions of the Potts model which arises in statistical mechanics. Despite having its origins in physics, the applications of the blob algebra are vast throughout pure mathematics. It brings together algebra and geometry and more recently categorical, diagrammatic and knot theoretic ideas. The blob algebra controls a portion of the representation theory of the Kac-Moody quantum algebra $U_{q}(\hat{\mathfrak{sl}}_2)$ as Martin and Ryom-Hansen \cite{MaRH} established via Ringel duality. More recently Iohara--Lehrer--Zhang \cite{IoLeZh} connected the Temperley-Lieb algebra of type B and the quantum algebra $U_{q}(\hat{\mathfrak{sl}}_2)$ via Schur--Weyl duality. Vaughan Jones used the Temperley-Lieb of type A to introduce the first knot invariant to distinguish between left and right trefoils, the well known \textit{Jones' polynomials}. These polynomials were introduced by studying the Temperley--Lieb algebras. 

Apart from abundance of applications, the blob algebra has a very fruitful representation theory and it is of great importance and interest itself. The blob algebra is a quotient of the Hecke algebra of type B, hence it controls portion of the representation theory of the affine Hecke algebras. The decomposition numbers for the ungraded blob algebra were determined by Martin and Woodcock \cite{MaWo2} and Ryom-Hansen \cite{RH}. However the fact that the blob algebra is graded caused a blast in its study. In their pioneering work, Brundan and Kleshchev \cite{BrKl} have proven that the cyclotomic Hecke algebras are isomorphic to the Khovanov-Lauda-Rouquier (KLR) algebras. The blob algebra $\blob$, $\kappa = (\kappa_1,\kappa_2)$ a bicharge, is isomorphic to a quotient of the KLR algebra of level 2. In particular it is isomorphic to the algebra with generators  
\[
\{\mathsf{e}(\ui) \ | \ \ui = (i_1,\cdots,i_{d})\}\cup\{\psi_{1},\cdots,\psi_{d-1}\}\cup\{y_{1},\cdots,y_{d}\}
\]
subject to KLR relations and two additional \textit{blob} relations. In particular the blob relations are the following
\[
y_{1}\mathsf{e}(\ui) = 0, \ \text{if} \ i_1 = \kappa_1,\kappa_2 \ \ \text{and} \ \ \mathsf{e}(\ui) = 0,  \ \text{if} \ i_2 = i_1 + 1.
\]
Plaza and Ryom--Hansen \cite{PlaRH} proved that the blob algebra admits a $\mathbb{Z}$-grading and also constructed a graded cellular basis. Namely the blob algebra $\blob$ is a graded cellular algebra with graded cellular basis indexed by standard tableaux of one-column bipartitions of $d$. The graded decomposition numbers for the blob algebra were computed by Plaza \cite{Pla}. Moreover the blob algebra is quasi-hereditary and in that setting Hazi, Martin and Parker \cite{HaMaPa} determined the structure of the indecomposable tilting modules using the graded structure.   

Having their origins in Lie theory, \textit{alcove geometries} play an important role in the understanding of the representation theory of Hecke and KLR algebras. Martin and Woodcock \cite{MaWo1} were the pioneers of using such techniques in the study of the blob algebra. Roughly speaking the action of the affine Weyl group of type $\hat{A}_1$ divides the 2-dimensional Euclidean space into connected components, the \textit{alcoves}, separated by \textit{hyperplanes} or \textit{walls} which we index by half integers. The alcoves are indexed by integers and the alcove $\mathfrak{a}_{m}$ is between the hyperplanes $H_{m-1/2}$, $H_{m+1/2}$. Each one-column bipartition $\la$ can be embedded in the Euclidean space and there is a natural bijection between standard $\la$-tableaux and paths in the Euclidean space terminating at the point $\la$. 

\begin{thmA*}
Let $F$ be a field of characteristic zero and $\la$ be an one-column bipartition of $d$ . 
\begin{enumerate}
\item If $\la\in\mathfrak{a}_m$, $m \le 0$, we have that
\[
L(\la) = \mathrm{span}_{F}\left\{\psi_{\TT} \ {\scalefont{1.3} | \ \substack{\TT \ \text{does not intersect} \ H_{1/2} \ \text{and} \\ \text{does not last intersect} \ H_{m-1/2}}}\right\}.
\]
\item If $\la\in H_{m-1/2}$, $m\le 0$, we have that
\[
L(\la) = \mathrm{span}_{F}\{\psi_{\TT} \ | \ \TT \ \text{does not intersect} \ H_{1/2}\}.
\]
\end{enumerate}
\end{thmA*}

\begin{figure}[h]
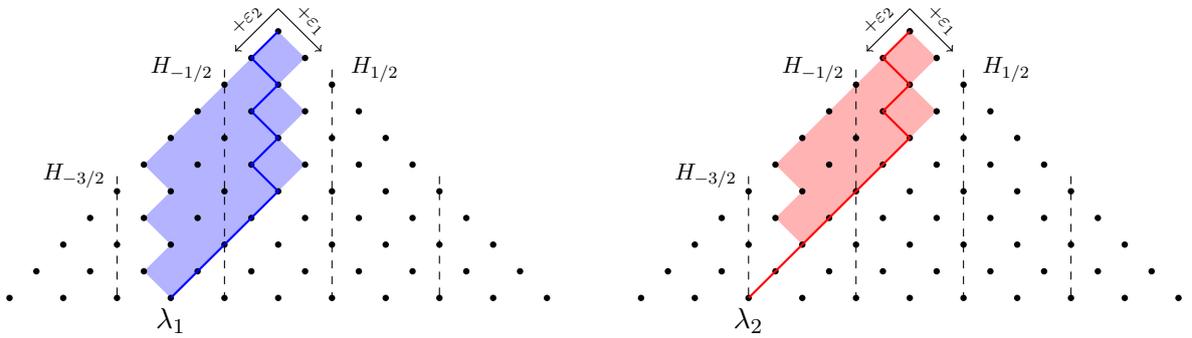

\[

 \]
 \vspace{-0.6cm}
  \caption{Here $\la_1 = ((1^{3}),(1^{7}))$, $\la_2 = ((1^2),(1^8))\in\Bip_{1}(10)$ are one-column bipartitions in an alcove and on a hyperplane respectively. The paths indexing basis elements for the simples $L(\la_{1})$ and $L(\la_{2})$ are in the shaded areas. The paths drawn are those corresponding to the initial tableaux $\mathsf{t}^{\la_{1}}$ and $\mathsf{t}^{\la_{2}}$ respectively (see Definition \ref{initial tableau}).}
   \label{intro eg}
\end{figure}

If $m > 0$ the result is similar up to relabelling hyperplanes and it is given in detail into the paper. In order to approach the problem we use the graded cellular structure of the blob algebra known from \cite{PlaRH}. We also develop representation theoretic methods by constructing a presentation for the cell modules $\Delta(\la)$. In more detail we identify \textit{Garnir relations} which are sufficient to describe the cell module.

In 1975 \cite{BGG} Bernstein--Gelfand--Gelfand constructed resolutions of simple modules by Verma modules in the context of finite-dimensional Lie algebras. Those resolutions, known as BGG resolutions, have applications in many areas of mathematics. In the study of the Laplacian space \cite{East}, complex representation theory of Kac-Moody algebras \cite{GaLe},  algebraic geometry \cite{EnHu}... In the context of modular representation theory of the symmetric group and Hecke algebras, BGG resolutions were first used by Bowman, Norton and Simental \cite{BoNoSi}. They utilised resolutions of Specht modules in order to provide homological construction of unitary simple modules of Cherednik and Hecke algebras of type A. They used these results in order to calculate Betti numbers and Castelnuovo--Mumford regularity of symmetric linear subspaces and this is another application of BGG resolution in algebraic geometry. 

In this paper we generalise \cite{BoNoSi} by showing that all simple modules of the blob algebra admit BGG resolutions. In particular we construct resolutions of cell modules for each simple  $\blob$-module indexed by a bipartition which belongs to an alcove. Simple modules indexed by one-column bipartitions which belong to hyperplanes have much easier BGG resolutions and they are used in the proof of the second main theorem of this paper which is the following.

\begin{thmB*}
Let $F$ be a field of characteristic zero and $\la$ be an one-column bipartition which belongs to an alcove and 
\[
C_{\bullet}(\la) := \bigoplus_{\nu\trianglelefteq\la} \Delta(\nu)\langle |\ell(\nu)| - |\ell(\la)| \rangle.
\] 
The complex
\[
\xymatrix{0 \ar[r] & C_{\bullet}(\la)  \ar[r] & L(\la) \ar[r] & 0
}
\]
with differentials given by one-column homomorphisms is a BGG resolutions for the simple module $L(\la)$. In other words 
\[
H_{i}(C_{\bullet}(\la)) =
\begin{cases}
L(\la), \ \text{if} \ i = 0 \\
0, \hspace{0.7cm} \text{otherwise}.
\end{cases}
\]
\end{thmB*} 

The paper is organised as follows. In sections 1 we introduce combinatorics of bipartitions and tableaux that arise in the representation theories of KLR algebras. We then describe the alcove geometry of type $\hat{A}_1$ which is crucial for the understanding of the structure of the blob algebra and the ideas we use in order to approach our problem. These sections are enriched with many examples and we try to give a geometric interpretation for the majority of the combinatorial and algebraic concepts. In section 2 we give the definition of the blob algebra as a quotient of the KLR algebra and we establish its structure as graded cellular and quasi-hereditary algebra. In this section we also provide a presentation of the cell modules of the blob algebra, in terms of the Garnir relations. In section 3 we establish the basic tool for the construction of the bases of simple modules, namely homomorphisms between cell modules. The construction of those homomorphisms utilises the Garnir-type presentation of the cell modules from section 2. The most important construction of section 3 is the spanning set of the images of the aforementioned homomorphisms. The elements of those spanning sets will certainly belong to the radical of the cell module. Section 4 includes the first of the two main results of this paper. We prove that over a field of characteristic zero the spanning set of the images is a basis for the radical of the cell $\blob$-module, hence we prove Theorem A. In section 5 we provide a homological construction of simple modules via the BGG resolutions and hence prove Theorem B. Again this construction is over a field of characteristic zero. 

\begin{ackn}
The author would like to thank his Ph.D. supervisor Dr C. Bowman for the helpful discussions and suggestions during this project. Moreover, the author would like to thank the referees for the comments and corrections, which contributed a lot towards the improvement of the content and the appearance of the paper.   
\end{ackn}

\renewcommand{\res}{\mathsf{res}}
\section{Combinatorics of tableaux and paths}

\subsection{Partitions and tableaux}
\renewcommand{\TT}{\mathsf{t}}
\renewcommand{\SS}{\mathsf{s}}
\renewcommand{\tT}{\mathsf{T}}
\renewcommand{\sS}{\mathsf{S}}
We fix two positive integers $d > 0$ and $e\in\{2,3,\cdots\}$. Let $\mathfrak{S}_d$ be the symmetric group in $d$ letters, with length function $\mathsf{L}$ and set $I := \mathbb{Z}/e\mathbb{Z}$. As Coxeter group $\mathfrak{S}_{d}$ is generated by the simple transpositions $s_1,\cdots,s_{d-1}$ subject to the relations
\begin{align*}
s_{i}^{2} & = 1 \hspace{1.5cm} \text{for} \ i = 1,\cdots,d-1 \\
s_{i}s_{j} & = s_{j}s_{i} \hspace{1cm} \text{for} \ i = 1\le i < j-1\le d-2 \\
s_{i}s_{i+1}s_{i} & = s_{i+1}s_{i}s_{i+1} \hspace{0.4cm} \text{for} \ i = 1,\cdots,d-2.
\end{align*}
We refer to the last two relations as \textsf{braid relations}. An \textsf{admissible $e$-bicharge} is a pair $\kappa = (\kappa_1,\kappa_2)\in \mathbb{Z}^2$ such that $0<|\kappa_1-\kappa_2|<e$. In this paper we make the convention that $\kappa_{1} < \kappa_{2}$. Note that this does not restrict us on the definition of the algebra, as the blob algebra is defined for any choice of admissible bicharge. This is just a convention about the combinatorics of the algebra. A \textsf{bipartition} of the positive integer $d$ is a pair of partitions $\lambda = (\lambda^{(1)},\lambda^{(2)})$ such that
\[
d = |\lambda^{(1)}| + |\lambda^{(2)}|
\]
and we denote by $\Bip(d)$ the set of bipartitions of $d$. The \textsf{diagram} of the bipartition $\lambda$ is the set
\[
[\lambda] := \{(r,c,m)\in\mathbb{N}\times\mathbb{N}\times\{1,2\} \ | \ 1\le c \le \la_{r}^{(m)}\} 
\]
where $\lambda_{r}^{(m)}$ is the $r^{\mathrm{th}}$ part of the partition $\lambda^{(m)}$, $m = 1,2$. The triples $(r,c,m)$ are called \textsf{nodes} or \textsf{boxes} and by using the usual convention we can think of the diagram as two arrays of boxes in the plane. A \textsf{$\la$-tableau} is a bijection $\TT\colon [\la]\lxr \{1,\cdots, d\}$, we say that the tableau $\TT$ has shape $\la$ and we write $\Shape(\TT) = \la$. We can think a $\la$-tableau as a diagram of $\la$, where the nodes are occupied by the integers $\{1,\cdots,d\}$. We denote by $\TT^{-1}(k)$ the node occupied by the integer $k\in\{1,\cdots,d\}$ and by $\TT(r,c,m)$ the integer occupying the node $(r,c,m)\in[\la]$. A tableau $\TT$ is called \textsf{standard} if the entries increase along rows and down columns in both components. We denote by $\Std(\la)$ the set of standard $\la$-tableaux and set
\[
\Std(d) := \bigcup_{\la\in\Bip(d)}\Std(\la).
\] 

\renewcommand{\Bip}{\mathsf{Bip}_{1}}   

\begin{conv}
Throughout this paper we shall be exclusively interested in one-column bipartitions, that is bipartitions of the form $\la = ((1^{\la_1}),(1^{\la_2}))$. By the notion \textit{bipartition} we shall always refer to one-column bipartitions and we shall denote the set of one-column bipartitions of $d$ by $\Bip(d)$. Moreover the nodes of the diagram of such bipartitions will be of the form $(r,1,m)$.
\end{conv}

\begin{rmk}
Consider the set $\Lambda_d = \{-d,-d+1,\cdots,d-1,d\}$. There is an obvious injective map between $\Lambda_d$ and the set $\Bip(d)$ of bipartitions of $d$, given by 
\[
\Bip(d) \lxr \Lambda_d, \ ((1^{\la_1}),(1^{\la_2}))\longmapsto \la_1 - \la_2.
\]
In other words we can identify each bipartition with an integer in the set $\Lambda_d$. Using the above bijection we freely identify a bipartition $((1^{\la_1}),(1^{\la_2}))$ and the integer $\la_1 - \la_2$.
\end{rmk}

Let $(r,1,m)\in [\la]$ be a node. We define the \textsf{content} of the node $(r,1,m)$ to be 
\[
\ct(r,1,m) := \kappa_m + 1 - r \in\mathbb{Z}
\] 
and the \textsf{residue} of the node $(r,1,m)$ as
\[
\res(r,1,m) := \ct(r,1,m) \ (\textsf{mod} \ e) \in I.
\]
We refer to a node of residue $i\in I$ as \textsf{$i$-node}.

\begin{defin}
Let $\la\in\Bip(d)$ and $\TT$ be a $\la$-tableau. We define the \textsf{residue sequence} of $\TT$ to be the $d$-tuple:
\[
\res(\TT) := (\res(\TT^{-1}(1)),\cdots,\res(\TT^{-1}(d))) \in I^d.
\]
\end{defin}

\begin{defin} \label{order on nodes}
Let $(r,1,m)$ and $(r',1,m')$ be two nodes. We define a \textsf{lexicographic order} and we write $(r,1,m) \preceq (r',1,m')$ if either
\begin{itemize}
\item[(i)] $\ct(r,1,m) < \ct(r',1,m')$ or 
\item[(ii)] $\ct(r,1,m) = \ct(r',1,m')$ and $m > m'$.
\end{itemize} 
If in addition $\mathsf{res}(r,1,m) = \mathsf{res}(r',1,m')$, we write $(r,1,m) \trianglelefteq (r',1,m')$ and we say that $(r,1,m)$ is less dominant than $(r',1,m')$. 
\end{defin}

Let $\la = ((1^{\la_1}),(1^{\la_2})), \mu = ((1^{\mu_1}),(1^{\mu_2}))\in\Bip(d)$ be two bipartitions of $d$. We write $\la\preceq\mu$ if $|\la_1 - \la_2| \ge |\mu_1 - \mu_2|$. We say that $\la$ is less dominant than $\mu$ and we write $\la\trianglelefteq\mu$, if additionally $\la, \mu$ have the same multiset of residues.

Now we shall generalise the order of Definition \ref{order on nodes} to an order on the set of standard tableaux. For a tableau $\TT$ we write $\TT{\downarrow_{\{1,\cdots,k\}}}$ for the subtableau of $\TT$ containing the entries $\{1,\cdots,k\}$ for $1\le k\le d$.

\begin{defin}\label{order on tableaux}
Let $\TT, \SS\in\Std(\la)$. We write $\TT \preceq \SS$ if and only if 
\[
\Shape(\TT{\downarrow_{\{1,\cdots,k\}}}) \preceq \Shape(\SS{\downarrow_{\{1,\cdots,k\}}}) 
\] 
for $1\le k \le d$. In addition if $\res(\TT) = \res(\SS)$, we write $\TT \trianglelefteq \SS$ and we say that $\TT$ is less dominant than $\SS$.
\end{defin}

Note that if $\TT\in\Std(d)$ is a standard tableau, we denote by $\TT^{t}$ the transpose of $\TT$. For ease of notation we will often use the transpose tableau. 

\begin{eg}
Let $\la = ((1),(1^{9}))\in\Bip(10)$, $\kappa = (0,2)$ and $e = 4$ and we consider the standard $\la$-tableaux 
\[
\Yvcentermath1
\Ylinethick{0.8pt}
{\TT}^{t} = \left( \, \young(\rten), \ \  \\
\young(123456789) \, \right) 
\]
and
\[
\Yvcentermath1
\Ylinethick{0.8pt}
{\SS}^{t} = \left( \, \young(8), \ \  \\
\young(12345679\rten) \, \right).
\]

We remark that $\res(\TT)\neq\res(\SS)$ and $\Shape(\TT{\downarrow_{\{1,\cdots,k\}}}) = \Shape(\SS{\downarrow_{\{1,\cdots,k\}}})$, for $1\le k\le 7$. However we have that 
$$
\Shape(\TT{\downarrow_{\{1,\cdots,k\}}}) \preceq \Shape(\SS{\downarrow_{\{1,\cdots,k\}}})
$$
for $8\le k\le 10$, hence we $\TT$ precedes $\SS$ in the order of Definition \ref{order on tableaux} and we write $\TT\prec\SS$. Now we consider the standard $\la$-tableau
\[
\Yvcentermath1
\Ylinethick{0.8pt}
{\mathsf{u}}^{t} = \left( \, \young(7), \ \  \\
\young(12345689\ten) \, \right) 
\]
and we note that $\res(\mathsf{u}) = \res(\SS)$. Since $\Shape(\SS{\downarrow_{\{1,\cdots,k\}}}) = \Shape(\mathsf{u}{\downarrow_{\{1,\cdots,k\}}})$, for $1\le k\le 6$, and 
$$
\Shape(\SS{\downarrow_{\{1,\cdots,k\}}}) \prec \Shape(\mathsf{u}{\downarrow_{\{1,\cdots,k\}}})
$$
for $7\le k\le 10$, we have that $\SS\trianglelefteq\mathsf{u}$, i.e $\SS$ is less dominant than $\mathsf{u}$.
\end{eg}

Another concept that we shall recall is the degree of a given standard tableau $\TT$. For that purpose we give a few more definitions. Let $\la\in\Bip(d)$ and $A$ be a node of $\la$. The node $A$ is called addable (resp. removable) if $[\la]\cup\{A\}\in\Bip(d+1)$ (resp. $[\la]-\{A\}\in\Bip(d-1)$) is a diagram of a bipartition. We emphasise that by the notion of addable node we also require that the resulting bipartition $[\la]\cup\{A\}$ will be an one-column bipartition of $d+1$. We denote by $\mathsf{Add}(\la)$ and $\mathsf{Rem}(\la)$ the set of addable and removable nodes of $\la$ respectively. For a residue $i\in I$ we define the sets
\[
\mathsf{Add}_{i}(\la) := \{A\in\mathsf{Add}(\la) \ | \ \res(A) = i \}\subset \mathsf{Add}(\la)
\]
and
\[
\mathsf{Rem}_{i}(\la) := \{A\in\mathsf{Rem}(\la) \ | \ \res(A) = i \}\subset \mathsf{Rem}(\la).
\] 
Then for a $\la$-tableau $\TT$ we denote by $\mathsf{Add}_{\TT}(k)$ and $\mathsf{Rem}_{\TT}(k)$ the following sets:
\begin{equation}\label{add of a tableau}
\mathsf{Add}_{\TT}(k) := \{A\in\mathsf{Add}(\Shape(\TT{\downarrow_{1,\cdots,k}})) \ | \ A\trianglelefteq \TT^{-1}(k)\}
\end{equation}
and 
\begin{equation}\label{rem of a tableau}
\mathsf{Rem}_{\TT}(k) := \{A\in\mathsf{Rem}(\Shape(\TT{\downarrow_{1,\cdots,k}})) \ | \ A\trianglelefteq \TT^{-1}(k)\}.
\end{equation}
for all $1\le k\le d$. By using (\ref{add of a tableau}), (\ref{rem of a tableau}) we define the degree of the tableau $\TT$.

\begin{defin}
Let $\TT\in\Std(d)$ be a standard tableau. We define the degree of the node $\TT^{-1}(k)$ to be
\[
\deg(\TT^{-1}(k)) := |\mathsf{Add}_{\TT}(k)| - |\mathsf{Rem}_{\TT}(k)|.
\]
The \textsf{degree of the tableau $\TT$} is the sum of the degrees of its nodes, namely
\[
\deg(\TT) = \sum_{k = 1}^{d}\deg(\TT^{-1}(k)).
\]
\end{defin}

\begin{defin}{\cite[Section 3]{Pla}}\label{initial tableau}
Let $\la = ((1^{\la_1}),(1^{\la_2}))\in\Bip(d)$ and $m = \min\{\la_1,\la_2\}$. We define the \textsf{initial tableau} $\TT^{\la}\in\Std(\la)$ to be the tableau obtained by filling the nodes increasingly down to columns as follows:
\begin{enumerate}
\item even numbers less than or equal to $2m$ in the first component,
\item odd numbers less than $2m$ in the second component,
\item numbers greater than $2m$ in the remaining nodes.
\end{enumerate}
\end{defin}

For a given bipartition $\la = ((1^{\la_1}),(1^{\la_2}))\in\Bip(d)$ the standard tableau $\TT^{\la}\in\mathsf{Std}(\la)$ is maximal under the order of Definition \ref{order on tableaux}. In order to simplify the notation, in later sections we shall write $\ui^{\la} = (i_{1}^{\la},\cdots,i_{d}^{\la})\in I^d$ instead of $\res(\TT^{\la})$ for the residue sequence of the tableau $\TT^{\la}$.

\begin{rmk}
The symmetric group $\mathfrak{S}_d$ acts in a natural way on the set  of tableaux. In particular if $\TT$ is a tableau and $s_{i}$ is a simple transposition, the tableau $s_i\TT$ obtained by interchanging the entries $i$, $i+1$. For any $\la$-tableau $\TT$ we define the word $w_{\TT}\in\mathfrak{S}_d$ to be the unique element of the symmetric group such that $w_{\TT}\TT^{\la} = \TT$.
\end{rmk}

\begin{defin}\label{skew}
Let $d, d^{\prime}\in\mathbb{Z}$ be two positive integers with $d^{\prime}< d$. If $\la\in\Bip(d)$ and $\nu\in\Bip(d^{\prime})$ with $[\nu]\subset [\la]$, we define the \textsf{skew bipartition} $\la\setminus\nu$ to be the bipartition with diagram the set difference $[\la]-[\nu]$.
\end{defin}

\begin{defin}
Let $d, d^{\prime}\in\mathbb{Z}$ with $d^{\prime}< d$, $\la\in\Bip(d)$, $\nu\in\Bip(d^{\prime})$ and let $\la\setminus\nu$ be the skew bipartition. If $\TT\in\Std(\nu)$ and $\SS\in\Std(\la\setminus\nu)$ then the $\la$-tableau with entries $\{1,2,\cdots,d\}$ in the nodes
\[
(\TT^{-1}(1),\cdots,\TT^{-1}(d^{\prime}),\SS^{-1}(1),\cdots,\SS^{-1}(d-d^{\prime})).
\]
respectively, is the composition $\TT\circ\SS\in\Std(\la)$ of the tableaux $\TT$ and $\SS$.
\end{defin}
   
\subsection{Paths and alcove geometry}

Let $\{\epsilon_1,\epsilon_2\}$ be formal symbols. We consider the 2-dimensional Euclidean space 
\[
V := \bigoplus_{i = 1, 2}\mathbb{R}\epsilon_i
\] 
with basis $\{\varepsilon_1,\varepsilon_2\}$ and let $V_{\mathbb{Z}_{\ge 0}}$ be the $\mathbb{Z}_{\ge 0}$-span of $\{\epsilon_1,\epsilon_2\}$. To any bipartition $\la = ((1^{\la_1}),(1^{\la_2}))\in\Bip(d)$ we attach a point of the Euclidean space $V$ via the embedding $((1^{\la_1}),(1^{\la_2}))\longmapsto \sum_{i = 1,2}\la_i\epsilon_i$. We consider the affine Weyl group $W_{\text{aff}}\cong\hat{\mathfrak{S}}_2$ of type $\hat{A}_1$ with $\alpha_1 = \epsilon_1 - \epsilon_2$ the corresponding simple real root. Note that the affine Weyl group is generated by the reflection $s_{\alpha_{1}, -1/2}$ and the reflection $s_{\alpha_{1}, 1/2}$, where the later corresponds to translation of the former by $e\alpha_{1}$. Let $(\cdot,\cdot)$ be a symmetric bilinear form on $V$ determined by $(\varepsilon_{i},\varepsilon_{j}) = \delta_{i,j}$, where $\delta_{i,j}$ is the Kronecker delta. For a given $e$-bicharge $\kappa = (\kappa_1,\kappa_2)$ we set $\rho := (\kappa_2 - \kappa_1)\varepsilon_1$. For any $m\in\mathbb{Z}$ we define the \textsf{hyperplane} 
\begin{equation}\label{hyperplanes}
H_{\alpha_1,m-\frac{1}{2}} := \{v\in V \ | \ (v+\rho, \alpha_1) = me\}.
\end{equation}
and we sometimes refer such a hyperplane as a \textsf{wall}. For any $m\in\mathbb{Z}$ there exists a unique reflection $s_{\alpha_1,m-1/2}$ such that 
\[
s_{\alpha_1,m-\frac{1}{2}}\cdot v = v - ((v+\rho,\alpha_1)-me)\alpha_1
\]
for any $v\in V$. In other words $s_{\alpha_1,m-1/2}$ acts on $V$ by reflection with respect to the hyperplane $H_{\alpha_1,m-1/2}$. From now on, since we have only one simple real root $\alpha_1$, we shall write simply $H_{m-1/2}$, $s_{m-1/2}$ for the wall and the reflection corresponding to the integer $m\in\mathbb{Z}$, respectively.

For any two integers $r,s\in\mathbb{Z}$ we denote by $[r,s]$ the set $[r,s] = \{t\in\mathbb{Z} \ | \ r\le t\le s\}$. For $d\in\mathbb{Z}_{>0}$ we define $\mathsf{Path}(d)$ to be the set of maps $\pi\colon [0,d]\lxr V_{\mathbb{Z}_{\ge 0}}$ such that 
\[
\pi(0) = 0 \ \text{and} \ \pi(k+1) - \pi(k)\in\{\epsilon_1,\epsilon_2\}
\]
for all $k\in [0,d-1]$ and we call its elements \textsf{paths} from $0$ to $d$. Given a standard tableau $\TT\in\mathsf{Std}(d)$ we define the point $\pi_{\TT}(k)$ in the space $V_{\mathbb{Z}_{}\ge 0}$ by the formula
\begin{equation}\label{step of path}
\pi_{\TT}(k) := c_{k,1}(\TT)\epsilon_1 + c_{k,2}(\TT)\epsilon_2
\end{equation} 
where $c_{k,i}(\TT)$ is the number of nodes of the tableau $\TT{\downarrow_{\{1,\cdots,k\}}}$ in the $i^{\mathrm{th}}$ component. Using the aforementioned notation we shall define the path in $V_{\mathbb{Z}_{\ge 0}}$ attached to a standard tableau $\TT\in\Std(\la)$.

\begin{defin}\label{path of a tableau}
Let $\TT\in\Std(d)$ be a standard tableau. We define the path $\pi_{\TT}$ corresponding to the tableau $\TT$ given by the sequence of points
\[
\pi_{\TT} = (\pi_\TT(0),\cdots,\pi_\TT(d))
\]
in the sense of relation (\ref{step of path}). There is a bijection between the set $\mathsf{Std}(d)$ of standard tableaux and the set of paths $\mathsf{Path}(d)$, given by $\TT\longmapsto\pi_{\TT}$.
\end{defin}



Using the notation above we shall define the reflected path through a hyperplane of $V$. 

\begin{defin}
Let $\TT\in\mathsf{Std}(d)$ and suppose that $\pi_{\TT}(a)\in H_{m-1/2}$ is the $i^{\mathrm{th}}$ intersection point of $\pi_{\TT}$ with the hyperplane $H_{m-1/2}$. We define the path $s_{m-1/2}^{i}\cdot\pi_\TT$ as follows
\[
(s_{m-\frac{1}{2}}^{i}\cdot\pi_\TT)(k) : =
\begin{cases}
\pi_\TT(k) \hspace{1.55cm} \text{if} \ 0\le k \le a \\
s_{m-1/2}\cdot\pi_\TT(k) \ \text{if} \ a < k\le d
\end{cases}.
\]   
We refer to the path $s_{m-1/2}^{i}\cdot\pi_\TT$ as the \textsf{reflected path} through the $i^{\mathrm{th}}$ intersection point of $\pi_\TT$ with the hyperplane $H_{m-1/2}$. 
\end{defin}

\begin{rmk}
Note that if the path $\pi_{\TT}$ intersects the hyperplane $H_{m-1/2}$ at a unique point, then we shall denote the reflected path simply by $s_{m-1/2}\cdot\pi_{\TT}$.
\end{rmk}

\begin{figure}
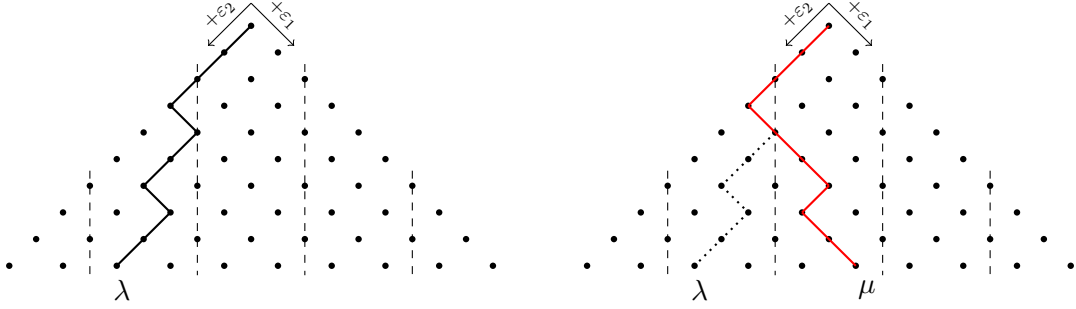

\[

 \]
 \vspace{-0.8cm}
  \caption{The path $\tT$ and the reflected path $s_{-1/2}^{2}\cdot\tT$.}
   \label{first examples of paths}
\end{figure}

In Figure \ref{first examples of paths} we visualise the last definition. We draw a path with endpoint the bipartition $\la$ and we also draw the reflected path through its second intersection point with the hyperplane $H_{-1/2}$.

Let $u,v\in V_{\mathbb{Z}_{\ge 0}}$ such that $u - v = \epsilon_i$, $i = 1, 2$. Then we define the degree of the pair $(u,v)$ as follows
\begin{equation}\label{degree of a pair}
\mathrm{deg}(u,v) := 
\begin{cases}
1 \hspace{0.465cm} \text{if} \ u\in H_{m-\frac{1}{2}} \ \text{and} \ |(v+\rho,\alpha_1)| < |me| \ \text{for some} \ m\in\mathbb{Z}; \\
-1 \ \text{if} \ v\in H_{m-\frac{1}{2}} \ \text{and} \ |(u+\rho,\alpha_1)| > |me| \ \text{for some} \ m\in\mathbb{Z}; \\
0 \hspace{0.47cm} \text{otherwise}.
\end{cases}
\end{equation} 

By using relation (\ref{degree of a pair}) we are able to give a reinterpretation of the degree of a tableau in terms of paths. Let $\TT\in\mathsf{Std}(d)$ and $\pi_{\TT}\in\mathsf{Path}(d)$ be the path corresponding to $\TT$. The integer 
\[
\mathrm{deg}(\pi_{\TT}) := \sum_{k = 0}^{d-1}\deg(\pi_{\TT}(k),\pi_{\TT}(k+1)) 
\]
is the \textsf{degree of the path} $\pi_{\TT}$. By \cite[Corrolary 4.6]{Pla} we have that the degree of the path $\pi_{\TT}$ coincides with the degree of the tableau $\TT$.

Using the aforementioned notions we are able to describe an \textsf{alcove geometry} on the Euclidean space $V$. We say that for any $m\in\mathbb{Z}$, the set of points
\[
\mathfrak{a}_m := \left\{v\in V \ | \ me < (v+\rho,\alpha_1) < (m+1)e\right\}
\]
forms an \textsf{alcove}. By the definition of the hyperplane as presented in (\ref{hyperplanes}), we can deduce that the origin, namely the point $(0,0)$, will always lie in an alcove and not on a hyperplane. We consider a Pascal triangle with points corresponding to integers and the top of the triangle corresponds to $0$. We can represent the paths in $V$ as paths in the Pascal triangle starting from the top and moving downwards. Let $\la = ((1^{\la_1}),(1^{\la_2}))\in\Bip(d)$ and $\TT\in\mathsf{Std}(\la)$ be a standard tableau. The path $\pi_{\TT}$ is a path starting from the top of the Pascal triangle and ending at a point corresponding to the integer $\la_1 - \la_2$ at the level $d$ of the triangle.

\begin{notation}
From now on we shall not distinguish between the standard tableau and the corresponding path. Namely, we will denote the path corresponding to the tableau $\TT$ by
\[
\tT = (\tT(0),\cdots,\tT(d))\in\mathsf{Path}(d).
\] 
Moreover, let $\la\in\Bip(d)$ and $\mu\in\Bip(d')$ with $d' < d$. We denote by $\mathsf{Path}(\mu\rightarrow\la)$ the set of paths starting from the bipartition $\mu$ and ending at the bipartition $\la$. We also let $\mathsf{Path}(\la) := \mathsf{Path}(\varnothing\rightarrow\la)$. Namely the paths of $\mathsf{Path}(\la)$ are paths starting from the top of the Pascal triangle and they have the bipartition $\la$ as endpoint. By using the above notation we have that
\[
\mathsf{Path}(d) = \bigcup_{\la\in\Bip(d)}\mathsf{Path}(\la).
\] 
\end{notation}

In the following example we shall summarise most of the facts we discussed above. Recall that we denote by $\TT^{t}$ the transpose of a tableau $\TT\in\Std(d)$.
  
\begin{eg}\label{eg on paths and reflection}
Suppose $e = 4$, $d = 9$, $\kappa = (0,2)$ and let $\la = ((1^2),(1^7)), \mu = ((1^5),(1^4))\in\Bip(9)$. We consider the $\la$-tableau
\[
\Yvcentermath1
\Ylinethick{0.8pt}
\TT^{t} = \left( \, \young(47), \ \  \\
\young(1235689) \, \right).
\] 

By following the description above we can construct the path corresponding to the tableau $\TT$ as in the left picture in Figure \ref{first examples of paths}. We observe that the path $\tT$ intersects the hyperplane $H_{-1/2}$ at two points which correspond to the steps $\mathfrak{t}_{-1/2}^{1}$ and $\mathfrak{t}_{-1/2}^{2}$ of the path. Then we obtain the reflected paths $s_{-1/2}^{1}\cdot\tT$, $s_{-1/2}^{2}\cdot\tT$ and the later is also pictured in Figure \ref{first examples of paths}. The endpoint of the reflected paths is the bipartition $\mu$.

Moreover one can easily calculate the degree of the path $\tT$ to be equal to $-1$. To see this note that $\mathrm{deg}(\tT(3),\tT(4)) = -1$ and degree is zero otherwise. This is something we expect since $\mathrm{deg}(\TT^{-1}(4)) = -1$ and the rest nodes of the tableau $\TT$ are of degree $0$.

The residue sequence of the tableau $\TT$ is 
\[
\mathsf{res}(\TT) = (2,1,0,0,3,2,3,1,0)
\] 
and we observe that $\mathsf{res}(s_{-{1}/{2}}^{2}\cdot\TT) = \mathsf{res}(\TT)$.  
\end{eg} 

More generally, from \cite[Lemma 4.7]{Pla} we have that given any two tableaux $\TT, \SS\in\mathsf{Std}(d)$ we have that
\begin{equation}\label{residues and reflections}
\mathsf{res}(\TT) = \mathsf{res}(\SS) \Longleftrightarrow \tT = s_{i_1-1/2}^{j_1}\cdots s_{i_a-1/2}^{j_a}\cdot\sS 
\end{equation}
for some simple reflections $s_{i_l-1/2}$, $1\le l\le a$. Given two bipartitions $\la,\mu\in\Bip(d)$ and $\tT\in\mathsf{Path}(\la)$, we define the set of $\mu$-paths which can be obtained by $\tT$ by a series of reflections as follows:
\begin{equation*}
\mathsf{Path}(\mu,\tT) := \{\sS\in\mathsf{Path}(\mu) \ | \ \sS =  s_{i_1-1/2}^{j_1}\cdots s_{i_a-1/2}^{j_a}\cdot\tT, \ \text{for some} \ s_{i_l-1/2}\in\hat{\mathfrak{S}}_2\}.
\end{equation*}
Now we equip our alcove geometry with a \textit{length function}  
\[
\ell\colon\Bip(d)\lxr\tfrac{1}{2}\mathbb{Z}, \ \la = ((1^{\la_1}),(1^{\la_2}))\longmapsto 
\begin{cases}
m & \text{if} \ \la_1\varepsilon_{1} + \la_2\varepsilon_{2} \in \mathfrak{a}_m \\
m-\frac{1}{2} & \text{if} \ \la_1\varepsilon_{1} + \la_2\varepsilon_{2} \in H_{m-\frac{1}{2}}.
\end{cases}
\]

We will also give a useful geometric interpretation of the dominance order on tableaux, mentioned in Definition \ref{order on tableaux} in terms of the alcove geometry. Given two tableau $\TT, \SS\in\mathsf{Std}(d)$ with $\mathsf{res}(\TT) = \mathsf{res}(\SS)$ we say that  the node $\TT^{-1}(k)$ is less dominant that the node $\SS^{-1}(k)$ in the sense of Definition \ref{order on nodes} if and only if 
\[
|\ell(\Shape(\TT\downarrow_{\{1,\cdots,k\}}))| > |\ell(\Shape(\SS\downarrow_{\{1,\cdots,k\}}))|.
\]

The tableau $\TT$ is less dominant  than $\SS$ if and only if $\TT^{-1}(k)\trianglelefteq \SS^{-1}(k)$, $1\le k\le d$, and there is at least one node of $\TT$ strictly less dominant than the corresponding node of $\SS$.   

\begin{eg}
We continue on the Example \ref{eg on paths and reflection} and we have that $\ell(\la) = -1$ while $\ell(\mu)$ = 0. The paths $\sS_1, \sS_2$ drawn in the following figure are the elements of $\mathsf{Path}(\mu,\tT)$.
\[

\]
\vspace{-1.1cm}

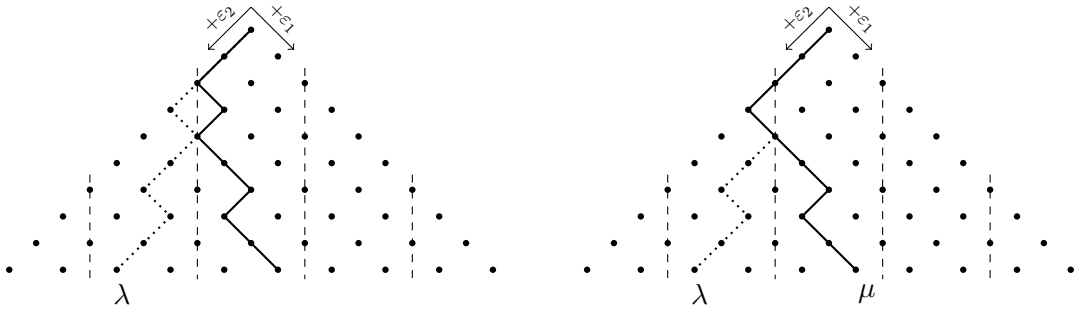
\captionof{figure}{The paths $\sS_1$ and $\sS_2$ are solid. The path $\tT$ is dotted.}
In particular we have that $\sS_1 = s_{-1/2}^{1}\cdot\tT$ and $\sS_2 = s_{-1/2}^{2}\cdot\tT$. Moreover we observe that $\sS_{1}\trianglerighteq\sS_{2}\trianglerighteq\tT$.
\end{eg}

Recall that to each bipartition $\la = ((1^{\la_1}),(1^{\la_2}))\in\Bip(d)$ we can attach the integer $\la_1 - \la_2$. Hence the action of the affine Weyl group $W_{\mathrm{aff}}$ on the set of bipartitions of $d$ can be described in terms of the action of $W_{\mathrm{aff}}$ on $\mathbb{Z}$. In particular, reflection $s_{m-1/2}$ corresponds to reflection about the integers $(\kappa_1 - \kappa_2) + me$, for all $m\in\mathbb{Z}$. Note that the integer obtained by reflecting as above corresponds to a bipartition of $d$. We say that two bipartitions $\la, \mu$ are \textit{linked} with respect to the alcove geometry of type $\hat{A}_1$ and we write $\la\sim\mu$ if they belong to the same $W_{\mathrm{aff}}$-orbit, i.e $\la\in W_{\text{aff}}\cdot\mu$. If $\tT\in\mathsf{Path}(\la)$ then the paths linked with $\tT$ are the paths of $\mathsf{Path}(\mu,\tT)$, defined above, for $\mu\sim\la$. The paths linked with the path $\tT^{\la}$ will be of particular interest when we construct homomorphisms of the blob algebra. We define 
\[
\mathsf{Path}_{\sim}(\la) := \bigcup_{\mu\trianglerighteq \la}\mathsf{Path}(\mu,\tT^{\la}).
\] 

\begin{rmk}
If $\la,\mu\in\Bip(d)$ are two bipartitions, we note that $\la$ is less dominant than $\mu$ if and only if $\la\sim\mu$ and $|\ell(\la)|>|\ell(\mu)|$, i.e. $\la$ is further away from the origin of the Pascal triangle than $\mu$.
\end{rmk}

Let us see an example regarding the notions we discussed above. 

\begin{eg}\label{example linkage principle}
Let $d = 9, e=4, \kappa = (0,2)$ as in Example \ref{eg on paths and reflection} and $\la = ((1),(1^{8}))\in\Bip(9)$. Then 
\[
\Yvcentermath1
\Ylinethick{0.8pt}
{(\TT^{\la})}^{t} = \left( \, \young(2), \ \  \\
\young(13456789) \, \right) 
\]
and the paths linked with $\tT^{\la}$ are drawn in the following diagram
\[

\]  
\vspace{-0.8cm}
\captionof{figure}{The red path is the path $\tT^\la$ and the black ones are those linked with $\tT^{\la}$.}
\vspace{0.1cm}
The new linked paths correspond to the tableaux
\[
\Yvcentermath1
\Ylinethick{0.8pt}
\SS_1^{t} = \left( \, \young(29), \ \  \\
\young(1345678) \, \right)
\]
\[
\Yvcentermath1
\Ylinethick{0.8pt}
\SS_2^{t} = \left( \, \young(25678), \ \  \\
\young(1349) \, \right) 
\]
and
\[
\Yvcentermath1
\Ylinethick{0.8pt}
\SS_3^{t} = \left( \, \young(256789), \ \  \\
\young(134) \, \right). 
\]
Thus $\mathsf{Path}_{\sim}(\la) = \{\tT^{\la}, \sS_1, \sS_2, \sS_3\}$ and one can easily see that for a given bipartition $\mu\in\Bip(9)$, $\mathsf{Path}(\mu,\TT^{\la})$ implies $\mu\trianglerighteq\la$.
\end{eg}

 
 

\section{The Blob Algebra}

\subsection{Definition and basic properties}
In this section we shall introduce the main object of our study, namely the blob algebra. The blob algebra was first introduced by Martin and Saleur \cite{MaSa}, but we shall present the equivalent definition given in \cite{PlaRH}.  We will not give many details regarding the structure of the blob algebra. The interested reader may look for further details in the literature, for example \cite{Pla} and \cite{PlaRH}. As in the last section we fix two positive integers $d > 0$ and $e\in\{2,3,\cdots\}$ with $I := (\mathbb{Z}/e\mathbb{Z})^d$ and let $F$ be a field of characteristic $p\ge 0$.

\begin{defin}\label{blob alg}
Let $\kappa = (\kappa_1,\kappa_2)$ be an $e$-bicharge. The \textsf{blob algbera} $\mathsf{B}_{d}^{\kappa}$ is the $F$-algebra on the generators
\[
\{\mathsf{e}(\ui) \ | \ \ui\in I^{d}\}\cup\{y_1,\cdots,y_d\}\cup\{\psi_1\cdots,\psi_{d-1}\}
\]
subject to the usual KLR relations
\begin{eqnarray}
\label{2.1}\mathsf{e}(\ui) & = & 0, \hspace{5.2cm} \text{if} \ i_1 \neq \kappa_1,\kappa_2 \\
\label{2.2}y_1\mathsf{e}(\ui) & = & 0, \hspace{5.2cm} \text{if} \ i_1 = \kappa_1,\kappa_2  \\
\label{2.3}\mathsf{e}(\ui)\mathsf{e}(\underline{j}) & = & \delta_{i,j}\mathsf{e}(\ui), \\
\sum_{\ui\in I^d}\mathsf{e}(\ui) & = & 1, \\
y_r\mathsf{e}(\ui) & = & \mathsf{e}(\ui)y_r, \\ 
\psi_r\mathsf{e}(\ui) & = & \mathsf{e}(s_r\ui)\psi_r, \\
y_ry_s & = & y_sy_r, \\
\label{2.8}\psi_ry_s & = & y_s\psi_r, \hspace{4.7cm} \text{if} \ |r-s|>1 \\
\label{2.9}\psi_r\psi_s & = & \psi_s\psi_r, \hspace{4.6cm} \text{if} \ |r-s|>1 \\
\label{2.10}\psi_ry_{r+1}\mathsf{e}(\ui) & = & (y_{r}\psi_{r}+\delta_{i_r,i_{r+1}})\mathsf{e}(\ui), \\
\label{2.11}y_{r+1}\psi_r\mathsf{e}(\ui) & = & (\psi_{r}y_{r}+\delta_{i_r,i_{r+1}})\mathsf{e}(\ui), \\
\label{2.12}\psi_{r}^{2}\mathsf{e}(\ui) & = & 
\begin{cases}
0, \hspace{4.93cm} \text{if} \ i_r = i_{r+1}  \\
\mathsf{e}(\ui), \hspace{4.53cm} \text{if} \ i_r \neq i_{r+1}\pm 1 \\
(y_{r+1} - y_r)\mathsf{e}(\ui), \hspace{2.68cm} \text{if} \ i_r = i_{r+1} + 1 \\
(y_r - y_{r+1})\mathsf{e}(\ui), \hspace{2.68cm} \text{if} \ i_r = i_{r+1} - 1 
\end{cases} \\
\psi_r\psi_{r+1}\psi_{r}\mathsf{e}(\ui) & = &
\begin{cases}
\label{2.13}(\psi_{r+1}\psi_r\psi_{r+1} + 1)\mathsf{e}(\ui), \hspace{1.57cm} \text{if} \ i_{r+2} = i_r = i_{r+1} - 1 \\
(\psi_{r+1}\psi_r\psi_{r+1} - 1)\mathsf{e}(\ui), \hspace{1.57cm} \text{if} \ i_{r+2} = i_r = i_{r+1} + 1 \\
\psi_{r+1}\psi_{r}\psi_{r+1}\mathsf{e}(\ui), \hspace{2.50cm} \text{otherwise} 
\end{cases}
\text{}
\end{eqnarray}
along with the additional blob relation
\begin{eqnarray}
\label{2.14}\mathsf{e}(\ui) & = & 0, \hspace{5cm} \text{if} \ i_2 = i_1 + 1.
\end{eqnarray}
\end{defin}

The function $\mathrm{deg}\colon\mathsf{B}^{\kappa}_{d}\lxr\mathbb{Z}$ determined by 
\[
\mathrm{deg}(\mathsf{e}(\ui)) = 0, \quad \mathrm{deg}(y_r\mathsf{e}(\ui)) = 2, \quad \mathrm{deg}(\psi_s\mathsf{e}(\ui)) =
\begin{cases} 
-2 \quad \text{if} \ i_s = i_{s+1} \\
0 \quad\quad\hspace{-0.1cm} \text{if} \ i_s\neq i_{s+1} \pm 1 \\
-1 \quad \text{if} \ i_s = i_{s+1} \pm 1
\end{cases}
\]
for $1\le r\le d$ and $1\le s\le d-1$ is a degree function of $\mathsf{B}_d^{\kappa}$. Thus the blob algebra is a $\mathbb{Z}$-graded algebra with $\mathrm{deg}$ being the degree function. We also let $\ast\colon\blob\lxr\blob$ be the anti-involution defined by fixing the KLR generators.

Note that there is a diagrammatic presentation of the blob algebra in terms of KLR diagrams. For a more detailed description about the KLR diagrams we refer to \cite{LibPl}, \cite{HaMaPa}. 
Each KLR diagram of $\blob$ consists of $d$ strings and each string carries a residue $i\in\mathbb{Z}/e\mathbb{Z}$. The bottom and the top of the KLR diagram are sequences of residues. The product of two KLR diagrams is given by vertical concatenation. If $\ui = (i_1, i_2,\cdots,i_d)\in I^{d}$ we have that that
\[

\]
The diagrammatic interpretation of (\ref{2.1})-(\ref{2.14}) is given in \cite{HaMaPa}.

If $w = s_{i_1}\cdots s_{i_l}\in\mathfrak{S}_d$ is a reduced expression of an element of the symmetric group, we set
\[
\psi_{w} = \psi_{i_1}\cdots\psi_{i_l}\in\mathsf{B}_{d}^{\kappa}.
\]
Recall that for any $\la$-tableau $\TT$, $\la\in\Bip(d)$, we have defined the reduced expression $w_{\TT} = s_{i_1}\cdots s_{i_l}\in\mathfrak{S}_d$ such that $\TT = w_{\TT}\TT^{\la}$. Later we shall prove that this reduced expression is unique up to the commuting relations of the symmetric group and the KLR algebra. We define the element
\[
\psi_{\TT} := \psi_{i_1}\cdots\psi_{i_l}\mathsf{e}(\ui^{\la}).
\]
and we shall refer to the word $w_{\TT}$ as \textsf{the reduced expression of $\TT$}.
\begin{defin}
Suppose that $\la\in\Bip(d)$ and $\TT, \SS\in\mathsf{Std}(\la)$. We define the element
\[
\psi_{\SS\TT} := \psi_{\SS}\mathsf{e}(\ui^{\la})\psi_{\TT}^{\ast}\in\mathsf{B}_{d}^{\kappa}.
\]
\end{defin}

\begin{thm}[{\cite[Theorem 6.10]{PlaRH}}]\label{blob cellular} The blob algebra $\mathsf{B}_{d}^{\kappa}$ is a graded $F$-algebra with basis
\[
\{\psi_{\SS\TT} \ | \ \SS, \TT\in\mathsf{Std}(\la) \ \text{for} \ \la\in\Bip(d)\}.
\]
We let $\mathsf{B}_{d,\triangleright\la}^{\kappa}$ be the $F$-submodule of $\blob$ with basis 
\[
\{\psi_{\mathsf{u}\mathsf{v}} \ | \ \mathsf{u},\mathsf{v} \in\mathsf{Std}(\mu) \ \text{for} \ \mu\in\Bip(d), \mu\triangleright \la\}.
\]
Under the anti-involution $\ast\colon\blob\lxr\blob$, we have $\psi_{\SS\TT}^{\ast} = \psi_{\TT\SS}$. For any $\la\in\Bip(d)$, $\TT\in\mathsf{Std}(\la)$ and $a\in\blob$ there exists $\alpha_{\mathsf{u}}\in F$ such that for all $\SS\in\mathsf{Std}(\la)$ 
\[
a\psi_{\SS\TT} = \sum_{\mathsf{u}\in\mathsf{Std}(\la)}\alpha_{\mathsf{u}}\psi_{\mathsf{u}\TT} \mod \ \mathsf{B}_{d,\triangleright\la}^{\kappa}.
\]
In particular the blob algebra $\blob$ is a graded cellular algebra. 
\end{thm}

\begin{rmk}\label{remark for the induction}

Let $\la\in\Bip(d)$ be a bipartition of $d$ and $r, r+1, r+2$, $1\le r\le d-2$, be three successive positive integers. There are eight different cases for a standard $\la$-tableau and four of them are depicted in Figure \ref{different cases of tableaux} and we denote them (T1)-(T4) respectively. The rest four standard tableaux are the ones obtained by interchanging the numbers between the components and we denote them (T1') - (T4'). For instance the tableau (T1') is the tableau with $r+2$ in the first component and $r, r+1$ in the second component.  
\begin{figure}[ht!]
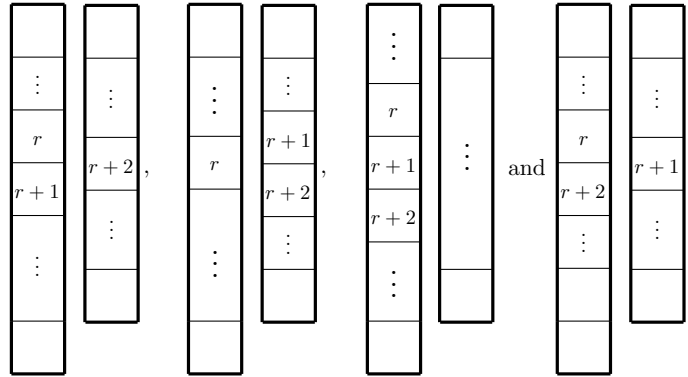

\[

\]
\vspace{-0.3cm}
\caption{We depict the four out of eight different cases of standard tableaux for three successive integers $r, r+1, r+2$ and we denote them (T1)-(T4).}
\label{different cases of tableaux}
\end{figure}

The subword $s_{r}s_{r+1}s_{r}$ cannot appear in the reduced expression of any of the above tableaux, as if we apply it to any standard tableau we get a non-standard tableau. In particular if $\TT\in\Std(\la)$ is the initial tableau, the non-standard tableau would be the one coming from the interchange of the nodes occupied by the entries $r, r+2$, which can be denoted by $\TT_{r\leftrightarrow r+2}$. Hence the reduced expression for each tableau is unique up to the commuting relations of the symmetric group. It follows that for any tableau $\TT$ with $w_{\TT} = s_{i_1}\cdots s_{i_k}$, the element $\psi_{i_1}\cdots\psi_{i_k}$ is unique up to the KLR relation (\ref{2.9}).

\end{rmk}

By the classical theory of cellular algebras as presented for example in \cite[Chapter 2]{Mat} we know that there exists a family of modules $\{\Delta(\la) \ | \ \la\in\Bip(\la)\}$ with $F$-basis
\[
\{\psi_{\TT}\ | \ \TT\in\mathsf{Std}(\la)\}
\]
called \textsf{cell} or \textsf{standard} modules and there is a unique bilinear form $\langle\cdot,\cdot\rangle\colon\Delta(\la)\times\Delta(\la)\lxr F$ such that $\langle\psi_{\SS},\psi_{\TT}\rangle$ for $\SS, \TT\in\mathsf{Std}(\la)$, is given by 
\[
\langle\psi_{\SS},\psi_{\TT}\rangle\psi_{\mathsf{u}\mathsf{v}} = \psi_{\mathsf{u}\SS}\psi_{\mathsf{t}\mathsf{v}} \mod \ \mathsf{B}_{d,\triangleright\la}^{\kappa}. 
\] 
The radical of a cell module $\Delta(\la)$ is given by
\[
\rad \ \Delta(\la) := \{x\in\Delta(\la) \ | \ \langle x,y\rangle = 0 \ \text{for all} \ y\in\Delta(\la)\}
\]
and let $L(\la) := \Delta(\la)/\rad \ \Delta(\la)$. By \cite[Section 9]{MaWo2}, the bilinear form is non-degenerate and so $\blob$ is quasi-hereditary with simples $\{L(\la) \ | \ \la\in\Bip(d)\}$.

\subsection{Graded decomposition numbers of the blob algebra}

As we discussed in the previous subsection the blob algebra $\blob$ is a $\mathbb{Z}$-graded algebra. Let $M$ be a finite dimensional graded $\blob$-module and let $M = \bigoplus_{i\in\mathbb{Z}}M_i$ be its decomposition into direct sum of homogeneous components. We define the graded dimension of $M$ to be the polynomial 
\[
\dim_{t}(M) := \sum_{i\in\mathbb{Z}}(\dim M_i)t^{i}\in\mathbb{Z}[t,t^{-1}]. 
\]
where $t$ is an indeterminate. Moreover if $L(\la)$ is a simple graded $\blob$-module then we denote by $L(\la)\langle k\rangle$ the graded $\blob$-module obtained by shifting the grading on $L(\la)$ up by $k$, namely 
\[
L(\la)\langle k\rangle =  \bigoplus_{i\in\mathbb{Z}}{L(\la)\langle k\rangle}_{i} = \bigoplus_{i\in\mathbb{Z}}L(\la)_{i-k}.
\]

The simple modules for the ungraded blob algebra were given by Martin and Woodcock \cite{MaWo2}. The following theorem summarises the work of Plaza and Ryom--Hansen using the Hu and Mathas' work on graded cellular structure of KLR algebras, \cite{PlaRH,HM}. 

\begin{thm}
Let $L(\la)$, $\la\in\Bip(d)$ be a simple module. Then 
\[
\{L(\la)\langle k\rangle \ | \ \la\in\Bip(d) \ \text{and} \ k\in\mathbb{Z}\}
\]
is a complete set of pairwise non-isomorphic simple graded $\blob$-modules.
\end{thm}

We also have the following useful proposition from \cite{HM}.

\begin{prop}\label{simples bar invariant}
For any $\la\in\Bip(d)$ the graded dimension of the simple module $L(\la)$ is bar-invariant (i.e., fixed under interchanging $t$ and $t^{-1}$).
\end{prop}

\begin{proof}
See \cite[Proposition 1.8]{HM}.
\end{proof}

It is important to know the block structure of the blob algebra. In our case of study the block structure is controlled by a \textit{linkage property} with respect to the affine Weyl group $W_{\mathrm{aff}}$ of type $\hat{A}_1$. 

\begin{prop}[{\cite[Theorem 9.3]{MaWo2}}]
Let $\la, \mu\in\Bip(d)$. Two simple modules $L(\la)$, $L(\mu)$ are in the same block of $\blob$ if and only if $\la$ and $\mu$ are in the same orbit, i.e $\la\in W_{\mathrm{aff}}\cdot\mu$.
 \end{prop}

For a graded $\blob$-module $M$ we denote by $[M\colon L(\la)\langle k\rangle]$ the graded multiplicity of the simple module $L(\la)\langle k\rangle$ as a graded composition factor of $M$. Then the graded decomposition number is
\[
[M\colon L(\la)]_{t} := \sum_{k\in\mathbb{Z}}[M\colon L(\la)\langle k\rangle]t^{k}\in\mathbb{Z}[t,t^{-1}].
\]
In particular we are interested in the decomposition matrix $\mathbf{D} = (d_{\mu\la})_{\mu,\la\in\Bip(d)}$, that is the decomposition numbers 
\[
d_{\mu\la} = [\Delta(\mu)\colon L(\la)]_t
\]
which were computed, over a field $F$ of characteristic zero, by Plaza \cite{Pla}. The closed formula for the graded decomposition number $[\Delta(\mu)\colon L(\la)]_t$ depends on whether the bipartition $\la$ lies in an alcove or on a hyperplane. However using the length function we defined before we can amalgamate the two distinct formulas into one. In what follows we assume that $\mu\trianglerighteq \la$, since this is the only case we can have a non-zero decomposition number, by Theorem \ref{blob cellular}. The following theorem gives the graded decomposition numbers of the blob algebra.   

\begin{thm}[{\cite[Theorem 5.11, 5.15]{Pla}}]\label{gr dec numbers}
 Let $\la,\mu\in\Bip(d)$ be two linked bipartitions with $\la\trianglelefteq\mu$. Then
\[
[\Delta(\mu)\colon L(\la)]_t = t^{|\ell(\la)|-|\ell(\mu)|}.
\]
\end{thm}

\begin{eg}
We continue with the example \ref{example linkage principle} and let $\la_1 = ((1^2),(1^7))$, $\la_2 = ((1^5),(1^4))$, $\la_3 = ((1^6),(1^3))$. According to Theorem \ref{gr dec numbers}, the (non-zero) graded decomposition numbers are the following:
\[
[\Delta(\la)\colon L(\la)] =  1, \ [\Delta(\la_1)\colon L(\la)] = t, \ [\Delta(\la_2)\colon L(\la)] = t^2 \ \text{and} \ [\Delta(\la_3)\colon L(\la)] = t.
\]
\end{eg}

\subsection{Garnir relations for the blob algebra}
In this subsection we provide a presentation for the cell modules of $\blob$. In particular we derive Garnir relations for the blob algebra. The above, apart from its importance on the structure of the cell modules, will be important in the next section where we wish to construct homomorphisms between cell modules. 

Recall from \cite{BrKlW} the Bruhat order $\le$ on $\mathfrak{S}_d$: for $u, w\in\mathfrak{S}_d$ we say that $u \le w$ if and only if some reduced expression of $u$ is a subexpression of  some reduced expression of $w$. By \cite[Theorem 5.10]{Hum} if $\underline{w}$ is a fixed, but arbitrary reduced expression for $w$, then $u\le w$  if and only if $u$ can be obtained as a subexpression of this reduced expression. We can connect the Bruhat order on permutations with the dominance order on tableaux.

\begin{lem}\label{Bruhat on tab}
Let $\la\in\Bip(d)$ and $\TT, \SS$ be two $\la$-tableaux with $w_\TT$, $w_\SS\in\mathfrak{S}_d$ be the unique permutations such that $w_{\TT}\TT^{\la} = \TT$ and $w_{\SS}\TT^{\la} = \SS$. Then
\[
\SS\trianglelefteq \TT \Longleftrightarrow w_\TT \le w_\SS.
\]
\end{lem}

\begin{proof}
The result is straightforward from the fact that $\TT = w_{\TT}\TT^{\la}$ and $\SS = w_{\SS}\TT^{\la}$. 
\end{proof}

\begin{defin}
Let $\la\in\Bip(d)$ and $A = (r,1,m)\in[\la]$ be a node of the diagram of $\la$. The node $A$ is called \textsf{Garnir node} if it is not removable.
\end{defin}

Suppose that $A = (r,1,m)\in[\la]$ is a Garnir node and let $u := \TT^{\la}(r,1,m)$ and $v := \TT^{\la}(r+1,1,m)$. It is clear from the definition of $\TT^{\la}$ that- for $(r,1,m)$ be Garnir node- there are two distinct cases for $u$ and $v$. In particular it will either be $v = u + 1$ or $v = u + 2$ and recall that $[u,v] := \{t\in\mathbb{Z} \ | \ u\le t\le v \}$. The \textsf{Garnir belt} $B^{A}$ is a set consisting of the nodes $(\TT^{\la})^{-1}(k)$ for $k\in [u,v]$.

\begin{defin}\label{Gardef}For a Garnir node $A = (r,1,m)\in[\la]$ we define a \textsf{Garnir tableau} $\mathsf{G}^A$ associated to $A$ to be the $\la$-tableau which:
\begin{itemize}
 \item coincides with $\TT^{\la}$ outside the Garnir belt $B^{A}$;
 \item has the numbers of the set $[u,v]$ in the remaining nodes according to the following rules.
\begin{enumerate}
\item If $v = u + 1$, then $\mathsf{G}^{A}$ has the numbers $u, u+1$ from the bottom to the top in the $m$th column;
\item if $v = u + 2$, then $\mathsf{G}^{A}$ has the entries $u, u+1, u+2$ from the bottom to the top in both components, first by filling one of the components and then by filling the other.
\end{enumerate}
\end{itemize}
\end{defin}

Note that we define \textit{a} Garnir tableau associated to $A$ rather that \textit{the} Garnir tableau, as the above definition does not always give a unique tableau. In the following remark we will clarify this point and we shall give a more concrete description of the Garnir tableaux.

\begin{rmk}\label{Garrmk}
Let $A = (r,1,m)$ be a Garnir node and $B^{A}$ be the Garnir belt associated to $A$.

When $v = u + 1$ there is a unique Garnir tableau $\mathsf{G}^{A}$, since there is a unique way of placing the numbers $u, u+1$. In particular, the tableau $\mathsf{G}^{A}$ is the tableau 
\begin{equation}\label{Garnir tab}
\mathsf{G}^{A} = s_{u}\TT^{\la}.
\end{equation}

When $v = u + 2$ there are two choices of Garnir tableaux. The first choice comes from filling the first component first and then the second component, while the second choice comes from filling the second component first and then the first component.  In particular, the two different Garnir tableaux are
\begin{equation}\label{Garntab1}
\mathsf{G}^{A}  = 
\begin{cases}
s_{u}s_{u+1}\TT^{\la} \\
s_{u+1}s_{u}\TT^{\la}.
\end{cases}
\end{equation}

\end{rmk}

The following example aims to clear the concept of the Garnir tableaux discussed in Definition \ref{Gardef} and Remark \ref{Garrmk}.

\begin{eg}
Let $d = 12$, $\kappa = (0,2)$ and $e = 4$. We consider the bipartition $\la = ((1^4),(1^8))\in\Bip(12)$ and the nodes $A = (2,1,1)$, $B = (6,1,2)\in[\la]$ which are Garnir nodes. The Garnir tableaux associated to $A$ and $B$ are the following non-standard tableaux

$$
\Yvcentermath1
\Ylinethick{0.8pt}
\mathsf{G}_{1}^{A} = \left( \, \young(2,!\blue5,4,!\wht8), \ \  \\
\young(1,3,!\blue6,!\wht7,9,\rten,\releven,\rtwelve) \, \right),
\Yvcentermath1
\Ylinethick{0.8pt}
\mathsf{G}_{2}^{A} = \left( \, \young(2,!\blue6,5,!\wht8), \ \  \\
\young(1,3,!\blue4,!\wht7,9,\rten,\releven,\rtwelve) \, \right) \ \text{and} \ 
\Yvcentermath1
\Ylinethick{0.8pt}
\mathsf{G}^{B} = \left( \, \young(2,4,6,8), \ \  \\
\young(1,3,5,7,9,!\blue\releven,\rten,!\wht\rtwelve) \, \right) 
$$
\end{eg}

where the nodes shaded in blue are the Garnir belts of each Garnir tableau. As expected, there are two distinct Garnir tableaux associated to the node $A$ and one unique Garnir tableau associated to $B$. One can easily check that $\mathsf{G}_{1}^{A} = s_{4}s_{5}\TT^{\la}$, $\mathsf{G}_{1}^{A} = s_{5}s_{4}\TT^{\la}$ and $\mathsf{G}^{B} = s_{10}\TT^{\la}$, as described in Remark \ref{Garrmk}.

In order to make the notation simpler we introduce the notion of the left and right exposed transposition. Let $\TT\in\Std(d)$ with reduced expression $w_{\TT} = s_{i_1}\cdots s_{i_k}$. A simple transposition $s_r$ is called \textsf{left exposed} (resp. \textsf{right exposed}) if $s_{r} = s_{i_j}$ for some $j\in\{1,\cdots,k\}$ and $s_r$ commutes with $s_{i_l}$ for all $l<j$ (resp. $l>j$).

\begin{lem}\label{non-standard factors}
Let $\la\in\Bip(d)$ and $\TT\not\in\Std(d)$. Suppose that $A = (r,1,m)\in[\la]$ is a node such that $\TT(r,1,m) > \TT(r+1,1,m)$. Then there exists $w\in\mathfrak{S}_{d}$ such that $\TT = w\mathsf{G}^{A^{\prime}}$ for some Garnir node $A^{\prime}\in[\la]$ and some Garnir tableau $\mathsf{G}^{A^{\prime}}$ and $\mathsf{L}(w_{\TT}) = \mathsf{L}(w) + \mathsf{L}(w_{\mathsf{G}^{A^{\prime}}})$. Conversely, if $\TT = w\mathsf{G}^{A^{\prime}}$ with $\mathsf{L}(w_{\TT}) = \mathsf{L}(w) + \mathsf{L}(w_{\mathsf{G}^{A^{\prime}}})$ then $\TT\not\in\Std(\la)$.
\end{lem}

\begin{proof}
Let $u := \TT^{\la}(r,1,m)$, $v := \TT^{\la}(r+1,1,m)$, $a := \TT(r,1,m)$ and $b := \TT(r+1,1,m)$. First we consider the case that $v = u + 1$. From our discussion in Remark \ref{Garrmk} we have that $\mathsf{G}^{A} = s_{u}\TT^{\la}$ and without loss of generality we may assume that $(r,1,m)$ is the unique node with $\TT(r,1,m) > \TT(r+1,1,m)$. If $\TT = \mathsf{G}^{A}$ we have nothing to prove, so let $\TT\neq \mathsf{G}^{A}$. If $a = v$ and $b = u$ the result is straightforward. Assume that $a\neq v$ and $b = a - 1$ and let $\SS := \TT_{a\leftrightarrow a-1}\in\mathsf{Std}(\la)$, hence $\TT = s_{a-1}\SS$. Then the word $s_{a-1}s_{a}s_{a-1}$ appears as subword of $w_{\TT}$ and by successively applying the braid Coxeter relations we end up with a subword of the form $s_{u}s_{u-1}s_{u}$ with $s_{u}$ being right exposed. Note that if $b\neq a-1$ then $\TT$ will be of the form 
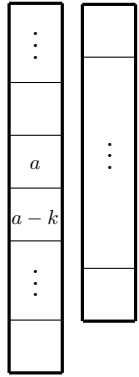
\[

\]
for some $2\le k\le a-1$ and we simply have the subword $s_{a-k}\cdots s_{a-2}$ on the left of $s_{u}s_{u-1}s_{u}$. 

In any case we have that 
\begin{equation}\label{r1}
\TT = w^{\prime}s_{u}s_{u-1}s_{u}\TT^{\la} = w^{\prime}s_{u}s_{u-1}\mathsf{G}^{A}
\end{equation}
for some permutation $w^{\prime}\in\mathfrak{S}_{d}$ and we have factorised the non-standard tableau $\TT$ through the unique Garnir tableau associated to $A$.

Now we consider the case that $v = u + 2$ and recall that the Garnir tableau associated to $A$ are the tableaux $\mathsf{G}_{1}^{A} := s_{u}s_{u+1}\TT^{\la}$ and $\mathsf{G}_{2}^{A} := s_{u+1}s_{u}\TT^{\la}$. Same as in the case $v = u+1$, if $\TT =\mathsf{G}_{1}^{A}$ or $\TT =\mathsf{G}_{2}^{A}$ the result is straightforward. Hence we may assume that $\TT\neq\mathsf{G}_{1}^{A}, \mathsf{G}_{2}^{A}$. If the entries $u, u+1, u+2$ occupy the nodes in $B^{A}$ in $\TT$ then the result is straightforward, that is $\TT= s_{u+1}\mathsf{G}_{1}^{A} = s_{u}\mathsf{G}_{2}^{A}$. Now suppose that the numbers $u, u+1, u+2$ do not occupy the nodes of $B^{A}$, but those nodes contain consecutive numbers $a, a+1, a+2$. Then if $a < u$ we have that on of the subwords $s_{a+3}s_{a+2}$ or $s_{a+2}s_{a+3}$ will appear in $w_{\TT}$ and it will be right exposed, hence
\begin{equation}\label{r2}
\TT = w\mathsf{G}^{B}, \ \text{where} \ B := (\TT^{\la})^{-1}(a+2)
\end{equation}
for some $w\in\mathfrak{S}_d$ and some Garnir tableau $\mathsf{G}^{B}$ associated to $B$. If $a > u$ then either $s_{a-2}s_{a-1}$ ir $s_{a-1}s_{a-2}$ will appear as subword of $w_{\TT}$ and it will be right exposed, hence
\begin{equation}\label{r3}
\TT = w\mathsf{G}^{C}, \ \text{where} \ C := (\TT^{\la})^{-1}(a-2)
\end{equation} 
for some $w\in\mathfrak{S}_d$ and some Garnir tableau $\mathsf{G}^{C}$ associated to $C$. From (\ref{r1}), (\ref{r2}) and (\ref{r3}) we have the desired result.

Conversely, suppose that $\TT\in\Std(\la)$. Since $\mathsf{G}^{A^{\prime}}$ is non-standard, we should have $\mathsf{L}(w_{\TT}) < \mathsf{L}(w) + \mathsf{L}(w_{\mathsf{G}^{A^{\prime}}})$ which is a contradiction. 
\end{proof}

\begin{thm}[{Relations for cell modules}]\label{G-relns for the blob}
Let $\la = ((1^{\la_1}), (1^{\la_2}))\in\Bip(d)$. Then 
\begin{eqnarray}
\label{2.21}\mathsf{e}(\ui)\psi_{\TT^{\la}} & = & \delta_{\ui,\ui^{\la}}\psi_{\TT^{\la}}, \ \delta_{\ui,\ui^{\la}} \  \text{the Kronecker delta,}  \\
\label{2.23}y_s\psi_{\TT^{\la}} & = & 0 \\
\label{2.22}\psi_r\psi_{\TT^{\la}} & = & 
\begin{cases}
\psi_{\TT^{\la}_{r\leftrightarrow r+1}}  & \text{if} \  \text{$r, r+1$ are in different components} \\
0 & \text{otherwise}
\end{cases} \\
\label{2.24}\psi_{t+1}\psi_{t}\psi_{\TT^{\la}} & = & 0 \\
\label{2.25}\psi_{t}\psi_{t+1}\psi_{\TT^{\la}} & = & 0
\end{eqnarray}
for all $1\le r\le d-1$, $1\le s\le d$ and $1\le t \le 2\min\{\la_1, \la_2\} - 2$. We refer to the relations (\ref{2.22})-(\ref{2.25}) as \textit{Garnir relations}. 
\end{thm}

\begin{proof}
Let $\mathsf{e}(\ui)$, $\ui\in I^{d}$ be a KLR idempotent of $\blob$. By the orthogonality relation we have that 
\[
\mathsf{e}(\ui)\psi_{\TT^{\la}} = \mathsf{e}(\ui)\mathsf{e}(\ui^{\la}) = 
\begin{cases}
\psi_{\TT^{\la}} & \text{if} \ \ui = \ui^{\la} \\
0 & \text{otherwise}.  
\end{cases} 
\]
The element $\psi_{r}\mathsf{e}(\ui^{\la})$ corresponds to a tableau with residue sequence 
\[
(i_{1}^{\la},\cdots,i_{r+1}^{\la},i_{r}^{\la},\cdots,i_{d}^{\la}).
\]
We use the fact that for any standard tableau $\TT$ the element $\psi_{w_{\TT}}$ is unique up to KLR relation (\ref{2.9}). If the nodes of $\TT^{\la}$ occupied by the entries $r$, $r+1$ are in the same component, then any tableau with such residue sequence indexes elements in the ideal $\mathsf{B}^{\kappa}_{d,\triangleright \la}$, hence $\psi_{r}\mathsf{e}(\ui^{\la}) = 0$ modulo more dominant terms. If they are in different components then the only choice for a tableau with the above residue sequence and corresponding permutation consisting of the generator $\psi_{r}$ is the tableau $\TT^{\la}_{r\leftrightarrow r+1}$, hence $\psi_{r}\mathsf{e}(\ui^{\la}) = \psi_{\TT^{\la}_{r\leftrightarrow r+1}}$. The element $y_{s}\mathsf{e}(\ui^{\la})$ corresponds to a tableau with residue sequence $\ui^{\la}\in I^{d}$. The unique tableau with that residue sequence is $\TT^{\la}$. However 
\[
\deg(y_s\mathsf{e}(\ui^{\la})) = 2 \neq 0 = \deg(\mathsf{e}(\ui^{\la}))
\]
thus $y_{s}\mathsf{e}(\ui^{\la})$. Regarding relation (\ref{2.24}), if $t, t+1$ are in the same component then the result follows from (\ref{2.22}). If $t$, $t+1$ are in different components then the element $\psi_{t+1}\psi_{t}\mathsf{e}(\ui^{\la})$ corresponds to a tableau with residue sequence 
\[
(i_{1}^{\la},\cdots,i_{t+1}^{\la},i_{t+2}^{\la},i_{t}^{\la},\cdots,i_{d}^{\la}).
\] 
But such standard $\la$-tableau does not exist hence $\psi_{t+1}\psi_{t}\mathsf{e}(\ui^{\la}) = 0$ modulo terms in the ideal $\mathsf{B}^{\kappa}_{d,\triangleright\la}$ (we note that there are $\mu$-tableaux of this residue sequence, for $\mu\triangleright\la$). Similarly we prove relation (\ref{2.25}).
\end{proof}

Now we are using the theory we developed in this section in order to get a presentation for the cell modules of $\blob$.

\begin{prop}\label{presentation for cell modules}
Let $\la\in\Bip(d)$. The generator $\psi_{\TT^{\la}}$ and relations of Theorem \ref{G-relns for the blob} form a presentation for the cell module $\Delta(\la)$.
\end{prop}   

\begin{proof}
By Theorem \ref{G-relns for the blob} we have that the desired relations are satisfied. By Lemma \ref{non-standard factors} and the fact that we know a basis for the cell modules, we deduce that this list of relations is complete. Hence this is enough for proving that the relations, together with the generator $\psi_{\TT^{\la}}$, form a presentation for the cell module $\Delta(\la)$.
\end{proof}

The connection we described in Lemma \ref{Bruhat on tab} between the Bruhat order on words and the dominance order on tableaux will be useful throughout this paper. In particular it can be used for proving technical results such as the following proposition which deals with the action of the KLR generators on specific elements of the cell module.

\begin{prop}\label{techn prop}
Let $\la\in\Bip(d)$ be a bipartition and $\TT\in\Std(\la)$ be a standard $\la$-tableau.
\begin{enumerate}
\item For any generator $y_r$, $1\le r\le d$, we have that
\[
y_r\psi_{\TT} = \sum_{\substack{\SS\in\Std(\la) \\ \SS\triangleright \TT}}\alpha_{\SS}\psi_{\SS}
\]
with $\alpha_{\SS}\in F$.
\item If $s_r\TT\not\in\mathsf{Std}(\la)$, then 
\[
\psi_r\psi_{\TT} = \sum_{\substack{\SS\in\Std(\la) \\ \SS\triangleright \TT \\ \mathsf{res}(\SS) = \mathsf{res}(s_r\TT)}}\alpha_{\SS}\psi_{\SS}
\]
with $\alpha_{\SS}\in F$.
\end{enumerate}
\end{prop}

\begin{proof} The result follows from \cite[Lemma 4.8, 4.9]{BrKlW}. \qedhere
\end{proof}

\section{Homomorphisms between cell modules}\label{Section 3}

\subsection{Construction of homomorphisms}
In this section we shall construct homomorphisms between certain cell modules of the blob algebra. By using the fact that $\blob$ is quasi-hereditary, we know that for a given bipartition $\nu\in\Bip(d)$ we have $\Hom_{\blob}(\Delta(\nu^{\prime}),\Delta(\nu))\neq 0$ only if $\nu^{\prime}\trianglelefteq\nu$. For the purposes of this paper we need to construct homomorphisms between cell modules indexed by linked bipartitions which also have lengths with absolute value differing by one.

Let $\mu\in\Bip(d)$ be a bipartition with $\ell(\mu) = m$ or $\ell(\mu) = m - 1/2$ for some integer $m\in\mathbb{Z}$. Equivalently $\mu$ lies in the alcove $\mathfrak{a}_m$ or in the hyperplane $H_{m-1/2}$.
\begin{enumerate}
\item Suppose that $\ell(\mu) = m$, $m\in\mathbb{Z}$. There exist at most two bipartitions $\la$, $\lapr$ with $\la, \lapr\sim\mu$ satisfying $|\ell(\la)| = |\ell(\lapr)| = |\ell(\mu)| + 1$. We wish to distinguish between these bipartitions (when they exist).
\begin{itemize}
\item If $m \le 0$ is a non-positive integer then we let $\ell(\la) < 0$ and $\ell(\lapr) > 0$;
\item if $m > 0$ is a positive integer then we let $\ell(\la) > 0$ and $\ell(\lapr) < 0$.
\end{itemize}
In this case we shall construct maps in the sets $\Hom_{\blob}(\Delta(\la^{\prime}),\Delta(\mu))$ and $\Hom_{\blob}(\Delta(\la),\Delta(\mu))$.
\item Suppose that $\ell(\mu) = m - 1/2$, $m\in\mathbb{Z}$. In this case we fix the unique bipartition $\lapr$ with $\lapr\sim\mu$ and $|\ell(\lapr)| = |\ell(\mu)| + 1$, such that
\begin{itemize}
\item if $m\le 0$ is a non-positive integer then $\ell(\lapr) > 0$;
\item if $m > 0$ is a positive integer then $\ell(\lapr) < 0$.
\end{itemize}
In this case we shall construct map in the set $\Hom_{\blob}(\Delta(\la^{\prime}),\Delta(\mu))$.
\end{enumerate}

\begin{notation}
From now on we will make the following abuse of notation. We shall not distinguish between the tableau $\TT$ and the attached path $\tT$ and both will be denoted by $\tT$. Moreover we shall denote by $\mathsf{t}_{m-1/2}^{i}$ the $i^{\mathrm{th}}$ intersection point of the path $\tT$ with the hyperplane $H_{m-1/2}$.  
\end{notation}

\renewcommand{\TT}{\mathsf{T}}
\renewcommand{\SS}{\mathsf{S}}
\renewcommand{\tT}{\mathsf{t}}

In what follows we shall restrict ourselves in the case that $m \le 0$ and we shall construct the maps we discussed above. Note that the results are not affected by whether $m\le 0$ or $m > 0$. This is just a convention in order to save space since everything is analogous for $m > 0$.   

\begin{defin}\label{homomorphisms}
Let $\mu \in\Bip(d)$ be a bipartition. 
\begin{enumerate}
\item If $\ell(\mu) = m$, $m\le 0$, we define the maps $\varphi_{\la}^{\mu}\colon \Delta(\la)\lxr\Delta(\mu)$ and $\varphi_{\lapr}^{\mu}\colon \Delta(\lapr)\lxr\Delta(\mu)$ as follows: 
\begin{equation}\label{hom1}
\varphi_{\la}^{\mu}(\psi_{\TT^{\la}}) := \psi_{s_{m-1/2}\cdot\TT^{\la}}
\end{equation} 
and
\begin{equation}\label{hom2}
\varphi_{\lapr}^{\mu}(\psi_{\TT^{\lapr}}) := \psi_{s_{1/2}\cdot\TT^{\lapr}}
\end{equation}
where the paths $s_{m-1/2}\cdot\TT^{\la}$ and $s_{1/2}\cdot\TT^{\lapr}$ are the reflections of the paths $\TT^{\la}$, $\TT^{\lapr}$ through the hyperplanes $H_{m-1/2}$, $H_{1/2}$ respectively. 
\item If $\ell(\mu) = m - 1/2$, $m\le 0$, then we define the map $\varphi_{\lapr}^{\mu}\colon\Delta(\lapr)\lxr\Delta(\mu)$ on the same way as in equation (\ref{hom2}).
\end{enumerate}
\end{defin}

\begin{rmk}\label{point of reflection of generator}
Note that each of the paths $\TT^{\la}$ and $\TT^{\lapr}$ intersects the hyperplanes $H_{m-1/2}$ and $H_{1/2}$ at precisely one point and we have dropped the superscripts.
\end{rmk}

\begin{rmk}\label{dual homs}
In the case that $m > 0$, we can define the maps $\varphi_{\la}^{\mu}$ and $\varphi_{\lapr}^{\mu}$ in an analogous way. In particular, if $\ell(\mu) = m > 0$, then  
\[
\varphi_{\la}^{\mu}(\psi_{\TT^{\la}}) := \psi_{s_{m+1/2}\cdot\TT^{\la}} \ \text{and} \ \varphi_{\lapr}^{\mu}(\psi_{\TT^{\lapr}}) := \psi_{s_{-1/2}\cdot\TT^{\lapr}}
\] 
where $s_{m+1/2}\cdot\TT^{\la}$ and $s_{-1/2}\cdot\TT^{\lapr}$ are the reflections of the paths $\TT^{\la}$, $\TT^{\lapr}$ through the hyperplanes $H_{m+1/2}$, $H_{-1/2}$ respectively. If $\ell(\mu) = m - 1/2$ then the desired map is $\varphi_{\la}^{\mu}(\psi_{\TT^{\lapr}}) = \psi_{s_{-1/2}\cdot\TT^{\lapr}}$. 
\end{rmk}

In the next proposition we shall prove that the maps of Definition \ref{homomorphisms} are indeed $\blob$-module homomorphisms. We cover the $m\le 0$ case, since the other case works analogously.

\begin{prop}\label{proof homs}
Let $\mu\in\Bip(d)$ with $\ell(\mu) = m$, $m\le 0$. The maps $\varphi_{\la}^{\mu}\colon \Delta(\la)\lxr\Delta(\mu)$ and $\varphi_{\lapr}^{\mu}\colon \Delta(\lapr)\lxr\Delta(\mu)$ of Definition \ref{homomorphisms} are homomorphisms of $\blob$-modules. 
\end{prop}

\begin{proof}
We shall prove the result for the map $\varphi_{\la}^{\mu}$ and then similar arguments apply for the map $\varphi_{\lapr}^{\mu}$. By Proposition \ref{presentation for cell modules} we need to show that the relations (\ref{2.22})-(\ref{2.25}) of Theorem \ref{G-relns for the blob} are satisfied. Recall that $\varphi_{\la}^{\mu}(\psi_{\TT^{\la}}) = \psi_{s_{m-1/2}\cdot\TT^{\la}}$ and let us denote $\SS := s_{m-1/2}\cdot\TT^{\la}$. Also let $\SS(n) = \TT^{\la}(n)\in H_{m-1/2}$, for some $1\le n\le d-1$, be the unique reflection point of the path $\TT^{\la}$ through the hyperplane $H_{m-1/2}$. Then 
\begin{eqnarray*}
\mathsf{e}(\ui)\varphi_{\la}^{\mu}(\psi_{\TT^{\la}}) &  = & \mathsf{e}(\ui)\psi_{\SS} \\
& = &  \psi_{w_{\SS}}\mathsf{e}(w_{\SS}^{-1}\ui)\mathsf{e}(\ui^{\mu}) \\
& = & \delta_{\ui,\mathsf{res}(\SS)}\psi_{w_{\SS}}\mathsf{e}(\ui^{\mu}) \\ 
& = & \delta_{\ui,\mathsf{res}(\varphi_{\la}^{\mu}(\psi_{\TT^{\la}}))}\varphi_{\la}^{\mu}(\psi_{\TT^{\la}})
\end{eqnarray*}
for any idempotent $\mathsf{e}(\ui)$ and so relation $(\ref{2.21})$ holds. Now consider the generator $y_{s}$ for some $1\le s\le d$. Then 
\[
y_{s}\varphi_{\la}^{\mu}(\psi_{\TT^{\la}}) = y_{s}\psi_{\SS} = 0
\]
since there does not exist element in $\mathsf{Std}(\mu)$ with residue sequence $\mathsf{res}(\SS)$ and degree equal to $\deg(\SS) + 2$. Hence relation (\ref{2.23}) holds.

Consider the element $\psi_{r}\varphi_{\la}^{\mu}(\psi_{\TT^{\la}})$, for some $1\le r\neq n < d$. then 
\begin{eqnarray*}
\psi_{r}\varphi_{\la}^{\mu}(\psi_{\TT^{\la}}) & = & \psi_{r}\psi_{\SS} \\ 
& = & \begin{cases} \psi_{\SS_{r\leftrightarrow r+1}} \hspace{0.2cm} \text{if} \ r, r+1 \ \text{are in different components} \\
0 \hspace{1.3cm} \text{otherwise}  \end{cases}  \\
& = & \begin{cases} {\varphi_{\la}^{\mu}(\psi_{\TT^{\la}})}_{r\leftrightarrow r+1} \hspace{0.2cm} \text{if} \ r, r+1 \ \text{are in different components} \\
0 \hspace{2.35cm} \text{otherwise}  \end{cases}
.
\end{eqnarray*}
In order to prove that relation (\ref{2.22}) holds we need to consider the case $r = n$. By construction, the simple transposition $s_{n}$ exists in $w_{\SS}$ and it is left exposed. Hence $\psi_{\SS} = \psi_{n}\psi_{i_1}\cdots\psi_{i_l}\mathsf{e}(\ui^{\mu})$. Since $\mathsf{res}(\SS^{-1}(n)) = \mathsf{res}(\SS^{-1}(n+1)) + 1$ (since we reflected through a hyperplane at this point) we have that
\begin{eqnarray*}
\psi_{n}\psi_{\SS} & = & \psi_{n}^{2}\psi_{s_{n}\SS} \\ 
& = & (y_{n+1}-y_{n})\psi_{s_{n}\SS}.
\end{eqnarray*} 
But both summands are zero since there does not exist standard $\mu$-tableau with residue sequence $\mathsf{res}(s_{n}\SS)$ and degree $\deg(s_{n}\SS) + 2$. Thus $\psi_{n}\psi_{\SS} = 0$ and so relation (\ref{2.22}) holds.

Consider the product $\psi_{r+1}\psi_{r}\varphi_{\la}^{\mu}(\psi_{\TT^{\la}})$, for $r = \TT^{\la}(A)$ where $A$ is a Garnir node as in the statement of Theorem \ref{G-relns for the blob}.  Then
$$
\psi_{r+1}\psi_{r}\varphi_{\la}^{\mu}(\psi_{\TT^{\la}}) =  \psi_{r+1}\psi_{r}\psi_{\SS} = \psi_{r+1}\psi_{s_{r}(\SS)} 
$$
and we deduce that the product is zero (modulo terms in the ideal $\mathsf{B}^{\kappa}_{d,\triangleright\mu}$), since there does not exist standard $\mu$-tableau with residue sequence $\mathsf{res}(s_{r+1}s_{r}\SS)$. Hence relation (\ref{2.24}) is satisfied. Similarly we prove that relation (\ref{2.25}) is also satisfied. \qedhere
\end{proof}

In the following definition we define a type of paths which shall be useful in the next subsection when we will prove some of the main results of the paper.

\begin{defin}
Let $\TT\in\Std(d)$ be a standard tableau. The path $\TT$ is called \textsf{length increasing} if 
\[
|\ell(\Shape(\mathsf{T}{\downarrow_{\{1,\cdots,k\}}}))|\le |\ell(\Shape(\mathsf{T}{\downarrow_{\{1,\cdots,k+1\}}}))|
\]
for all $1\le k < d$.
\end{defin}

For a given bipartition $\mu\in\Bip(d)$, an example of a length increasing path is the following:
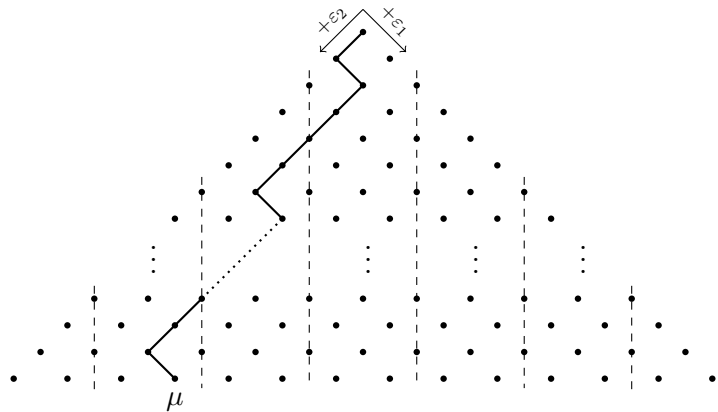
\[

\]

An alternative criterion to be a length increasing path of shape $\mu$ is that for a given bipartition $\mu\in\Bip(d)$ with $\ell(\mu) = m < 0$ (resp. $\ell(\mu) = m > 0$) every length increasing path in $\mathsf{Path}(\mu)$ intersects with the hyperplanes $H_{-1/2},\cdots, H_{m+1/2}$ (resp. $H_{1/2},\cdots, H_{m-1/2}$) at exactly one point and it does not intersect $H_{m+3/2},H_{m+5/2},\cdots$ (resp. $H_{m-3/2},H_{m-5/2},\cdots$).  In the exceptional case that $\ell(\mu) = 0$ we have that a length increasing path is a path that does not leave the fundamental alcove at any point.  

The following lemma will be useful in the sequel.

\begin{lem}\label{length incr is in simple}
Let $\nu\in\Bip(d)$ and $\TT\in\mathsf{Path}(\nu)$ be a length increasing path. Then the element $\psi_{\TT}$ belongs to the simple module $L(\nu)$.
\end{lem}

\begin{proof}
Let $\mathsf{res}(\TT)\in I^d$ be the residue sequence of $\TT$. Since $\TT$ is length increasing, the set $\mathsf{Path}(\nu^{\prime},\TT^{\nu})$ is non-empty only if $\nu^{\prime}\trianglerighteq\nu$. Hence we have that $\mathsf{e}(\mathsf{res}(\TT))\Delta(\nu^{\prime}) = 0$, for any bipartition $\nu^{\prime}\triangleleft \nu$. Thus $\mathsf{e}(\mathsf{res}(\TT))L(\nu^{\prime}) = 0$, for any bipartition $\nu^{\prime}\triangleleft \nu$. This shows that the element $\psi_{\TT}$ belongs to a composition factor of $\Delta(\mu)$ not of the form $L(\nu^{\prime})$, $\nu^{\prime}\triangleleft\nu$, so it belongs to the simple head $L(\nu)$.  
\end{proof}

\subsection{Image of the homomorphisms} In this subsection we shall construct the image of the homomorphisms $\varphi_{\lapr}^{\mu}$ and $\varphi_{\la}^{\mu}$ of Definition \ref{homomorphisms}. Same as in last subsection we cover the case that $m\le 0$, since all the arguments work equally in the case $m > 0$ up to relabelling hyperplanes. In the alcove case we compute the image of both $\varphi_{\la}^{\mu}$, $\varphi_{\lapr}^{\mu}$ whereas in the hyperplane case it is only necessary to consider the homomorphism $\varphi_{\la}^{\mu}$.  

Suppose that $\TT_{1}\in\mathsf{Path}(\lapr)$ is a length increasing path. The image of the element $\psi_{\TT_{1}}$ under the homomorphism $\varphi_{\lapr}^{\mu}$ is
\[
\varphi_{\lapr}^{\mu}(\psi_{\TT_{1}}) = \psi_{s_{1/2}\cdot\TT_{1}}
\]
since the path $s_{1/2}\cdot\TT_{1}$ is the unique path with residue sequence equal to $\mathsf{res}(\TT)$ terminating at the bipartition $\mu$. For the same reason, if $\ell(\mu) = m$, $m\le 0$ and $\TT_{2}\in\mathsf{Path}(\la)$ is a length increasing path then the image of the element $\psi_{\TT_{2}}$ under the homomorphism $\varphi_{\la}^{\mu}$ is 
\[
\varphi_{\la}^{\mu}(\psi_{\TT_{2}}) = \psi_{s_{m-1/2}\cdot\TT_{2}}.
\]
 
The following proposition is one of the main results of the section and describes a spanning set for the image of the homomorphism $\varphi_{\lapr}^{\mu}$. Note that the result holds for both $\ell(\mu) = m$ and $\ell(\mu) = m - 1/2$, $m\le 0$. 

\begin{prop}\label{image of lapr}
The homomorphism $\varphi_{\lapr}^{\mu}\colon\Delta(\lapr)\lxr\Delta(\mu)$ of Definition \ref{homomorphisms} is an injective homomorphism. Moreover 
\begin{enumerate}
\item if $m\le 0$
\[
\Image\varphi_{\lapr}^{\mu} = \mathrm{span}_{F}\{\psi_{\mathsf{U}} \ | \ \mathsf{U}\in\mathsf{Path}(\mu), \ \mathsf{U} \ \text{intersects} \ H_{1/2}\},
\]
\item if $m > 0$
\[
\Image\varphi_{\lapr}^{\mu} = \mathrm{span}_{F}\{\psi_{\mathsf{U}} \ | \ \mathsf{U}\in\mathsf{Path}(\mu), \ \mathsf{U} \ \text{intersects} \ H_{-1/2}\}.
\]
\end{enumerate}
\end{prop} 
\begin{proof}
We cover the case $m\le 0$ as the other one works similarly. Take any path $\mathsf{U}\in\mathsf{Path}(\mu)$ and suppose that it intersects the hyperplane $H_{1/2}$ at $n$-many points and let $\mathsf{u}_{1/2}^{n}$ be the final one. Then we notice that the reflection $s_{1/2}^{n}\cdot\mathsf{U}$ through the final point that $\mathsf{U}$ intersects the hyperplane $H_{1/2}$ gives a path terminating at $\lapr$. This shows that there is a bijection between the paths in $\mathsf{Path}(\mu)$ intersecting $H_{1/2}$ and the paths in $\mathsf{Path}(\lapr)$. We will prove that any path intersecting the hyperplane $H_{1/2}$ belongs indeed to the image of $\varphi_{\lapr}^{\mu}$ and thus the result will follow from the fact that the dimension of $\Delta(\lapr)$ is equal to the dimension of the image of $\varphi_{\lapr}^{\mu}$. 

We consider the path $\TT^{\mu}$ and we fix integer $a\in\mathbb{Z}$, $1\le a\le d-1$. Let $\TT\in\mathsf{Path}(\mu)$ be the maximal path, under the lexicographic order, with the property $\TT(a)\in H_{1/2}$ (see Figures \ref{Pasc tr1}, \ref{Pasc tr2}). We proceed by considering each value, $a$, one at a time. Since $\TT$ must intersect $H_{1/2}$ at some point by assumption, this allows us to consider all such paths. 
 
Note that the reflection $s_{1/2}\cdot\TT$ of the path $\TT$ through the hyperplane $H_{1/2}$ is a length increasing path in $\mathsf{Path}(\lapr)$. Hence the element $\psi_{\TT}$ belongs to the image of the homomorphism $\varphi_{\lapr}^{\mu}$. Let $\mathsf{U}\in\mathsf{Path}(\mu)$ be any path which intersects the hyperplane $H_{1/2}$ at the point $\mathsf{U}(a) = \TT(a)$, with $w_{\mathsf{U}} = s_{i_1}\cdots s_{i_k}\in\mathfrak{S}_d$ its reduced expression. Since $\TT$ is the maximal path, under the lexicographic order, with $\TT(a)\in H_{1/2}$ we have that $w_{\TT}$ is a subword of $w_{\mathsf{U}}$.  

\begin{figure} [h]
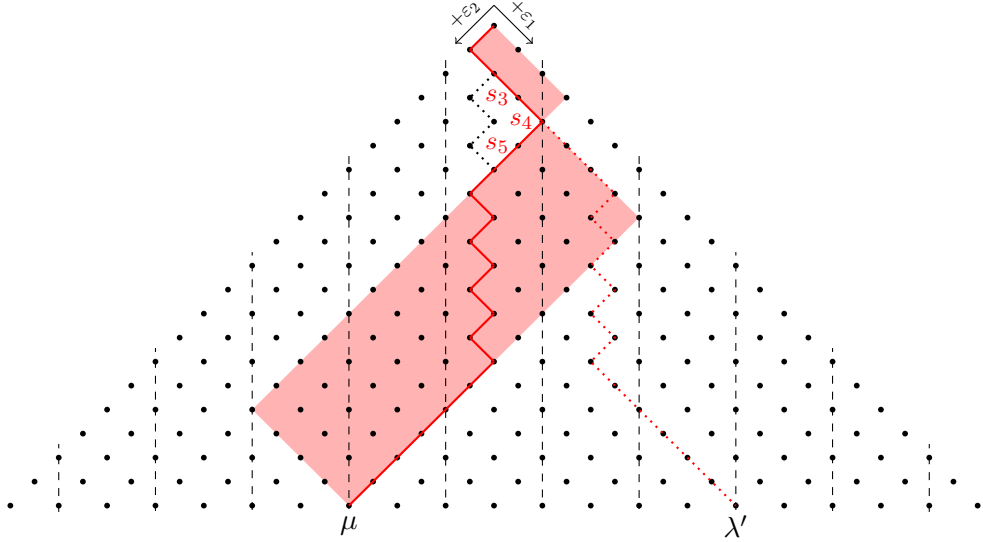

\[

\]
\vspace{-0.8cm}
\caption{For $d = 20$, $\mu = ((1^{7}),(1^{13}))$, $\ell(\mu) = -3/2$, $e = 4$, $\kappa = (0,2)$ and $a = 4$ the red path is the path $\TT$. The shaded area corresponds to all the paths that can be obtained from $\TT$ and the elements corresponding to them belong to the image of $\varphi_{\la}^{\mu}$. Any path within the shaded region has $s_{4}s_{3}s_{5}$ as a subword.}
\label{Pasc tr1}
\end{figure}
\vspace{0.1cm}

Note that the subword $w_{\TT}$ will be right exposed, as otherwise the condition $\mathsf{U}(a)\in H_{1/2}$ would not hold. We can rewrite the reduced expression $w_{\mathsf{U}}$ as 
\[
w_{\mathsf{U}} = s_{i_1}\cdots s_{i_l}w_{\TT}, \ 1\le l\le k
\] 
hence $\mathsf{U} = s_{i_1}\cdots s_{i_l}\TT$, $1\le l\le k$.

\begin{figure}[h]
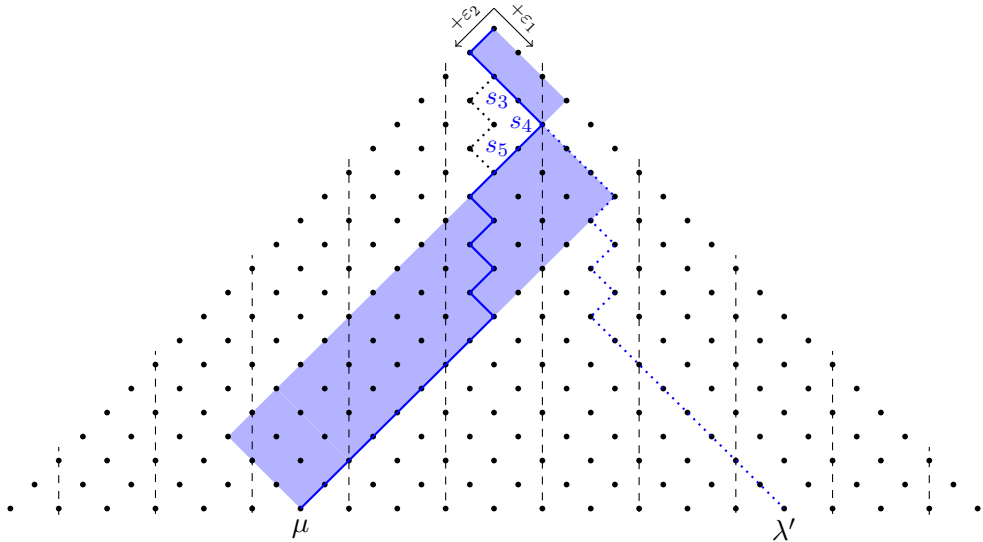

\[

\]
\vspace{-0.8cm}
\caption{For $d = 20$, $\mu = ((1^{6}),(1^{14}))$, $\ell(\mu) = -2$, $e = 4$, $\kappa = (0,2)$ and $a = 4$ the blue path is the path $\TT$. The shaded area corresponds to all the paths that can be obtained from $\TT$ and the elements corresponding to them belong to the image of $\varphi_{\la}^{\mu}$. Any path within the shaded region has $s_{4}s_{3}s_{5}$ as a subword.}
\label{Pasc tr2}
\end{figure}

Then the basis element $\psi_{\mathsf{U}}$ corresponding to the path $\mathsf{U}$ can be written as
\[
\psi_{\mathsf{U}} = \psi_{i_1}\cdots \psi_{i_l}\psi_{\TT} 
\]
and it belongs to the image of $\varphi_{\la}^{\mu}$, since $\psi_{\TT}$ does. By repeating the same procedure for all admissible integers $a\in\mathbb{Z}$, $1\le a\le d$, we prove that all paths in $\mathsf{Path}(\mu)$ which intersect the hyperplane $H_{1/2}$ correspond to elements in the image of the homomorphism $\varphi_{\la}^{\mu}$. \qedhere
\end{proof}

\begin{eg}
Let $d = 20$, $e = 4$, $\kappa = (0,2)$ and $\mu = ((1^7),(1^{13}))$. If $\TT = s_4s_3s_5\TT^{\mu}$ then we observe that the path 
\[
\SS = s_{15}s_{14}s_{16}s_{18}s_{13}s_{15}s_{17}s_{3}s_{5}s_{12}s_{14}s_{16}s_{2}s_{4}s_{6}s_{11}s_{13}s_{15}s_{1}s_{3}s_{5}s_{7}s_{10}s_{12}s_{14}\TT^{\mu}
\]
can be written as 
\begin{eqnarray*}
\SS & = & s_{15}s_{14}s_{16}s_{18}s_{13}s_{15}s_{17}s_{3}s_{5}s_{12}s_{14}s_{16}s_{2}s_{6}s_{11}s_{13}s_{15}s_{1}s_{7}s_{10}s_{12}s_{14}(s_{4}s_{3}s_{5}\TT^{\mu}) \\
& = & s_{15}s_{14}s_{16}s_{18}s_{13}s_{15}s_{17}s_{3}s_{5}s_{12}s_{14}s_{16}s_{2}s_{6}s_{11}s_{13}s_{15}s_{1}s_{7}s_{10}s_{12}s_{14}\TT
\end{eqnarray*}
by using the Coxeter relations of the symmetric group (see Figure \ref{Pasc tr3}).

\begin{figure}[ht!]
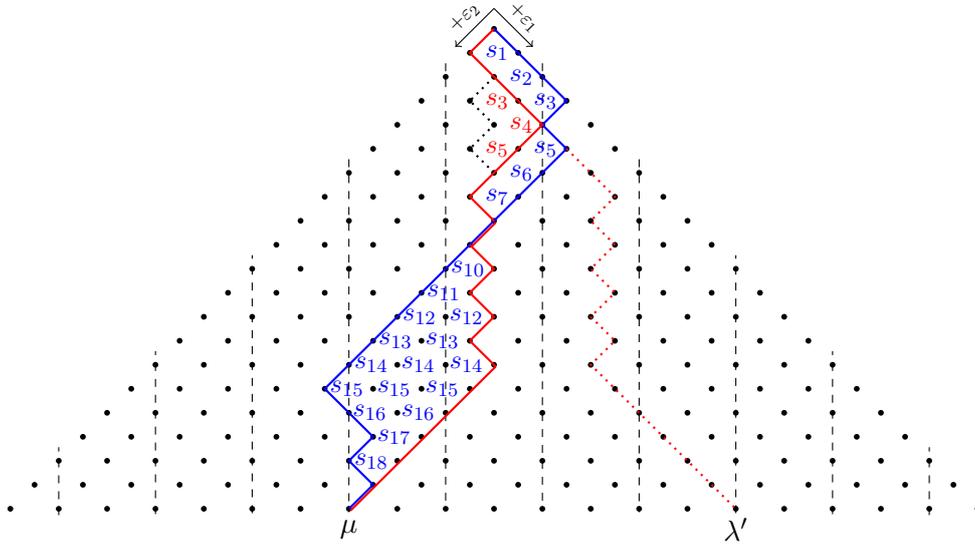

\[

\]
\vspace{-0.8cm}
\caption{The red path is the path $\TT$ intersecting the hyperplane $H_{1/2}$ at the point $\TT(4)$. The blue path is the path $\SS$ which intersects $H_{1/2}$ at $\SS(4) = \TT(4)$ and it belongs to the image of the homomorphism $\varphi_{\lapr}^{\mu}$.}
\label{Pasc tr3}
\end{figure}

Hence the basis element $\psi_{\SS}$ can be written as 
\[
\psi_{\SS} = \psi_{15}\psi_{14}\psi_{16}\psi_{18}\psi_{13}\psi_{15}\psi_{17}\psi_{3}\psi_{5}\psi_{12}\psi_{14}\psi_{16}\psi_{2}\psi_{6}\psi_{11}\psi_{13}\psi_{15}\psi_{1}\psi_{7}\psi_{10}\psi_{12}\psi_{14}\psi_{\TT}\in\Delta(\mu)
\]
and since the element $\psi_{\TT}$ belongs to the image of $\varphi_{\lapr}^{\mu}$, we have that $\psi_{\SS}$ also belongs to the image of $\varphi_{\lapr}^{\mu}$. 
\end{eg}

Recall that the homomorphism $\varphi_{\la}^{\mu}\colon\Delta(\la)\lxr\Delta(\mu)$ only exists when $\ell(\mu) = m$, that is the bipartition $\mu$ lies in the alcove $\mathfrak{a}_m$. The construction of the spanning set for the image of the homomorphism $\varphi_{\la}^{\mu}$ is the next important result of our paper towards our aim to construct bases for the irreducible representations of $\blob$.  For completeness we give the spanning sets for both $m\le 0$ and $m > 0$.

\begin{prop}\label{image of la}
The homomorphism $\varphi_{\la}^{\mu}\colon\Delta(\la)\lxr\Delta(\mu)$ of Definition \ref{homomorphisms} is an injective homomorphism. Moreover
\begin{enumerate}
\item if $m\le 0$
\[
\Image\varphi_{\la}^{\mu} = \mathrm{span}_{F}\left\{\psi_{\mathsf{U}} \ |{\scalefont{1.3} \substack{ \ \mathsf{U}\in\mathsf{Path}(\mu), \ \mathsf{U} \ \text{last intersects} \ H_{m-1/2} \ \text{or} \ \\ \text{intersects} \ H_{1/2} \ \text{after intersecting} \ H_{-1/2}}}\right\},
\]
\item if $m > 0$
\[
\Image\varphi_{\la}^{\mu} = \mathrm{span}_{F}\left\{\psi_{\mathsf{U}} \ |{\scalefont{1.3} \substack{ \ \mathsf{U}\in\mathsf{Path}(\mu), \ \mathsf{U} \ \text{last intersects} \ H_{m+1/2} \ \text{or} \ \\ \text{intersects} \ H_{-1/2} \ \text{after intersecting} \ H_{+1/2}}}\right\}.
\]
\end{enumerate} 
\end{prop}

Before presenting the proof, we shall give an example which illustrates which paths we are referring to in the statement of Proposition \ref{image of la}.

\begin{eg}
Let $d = 20$, $e = 4$, $\kappa = (0,2)$ and consider the bipartition $\mu = ((1^{6}),(1^{14}))$ with $\ell(\mu) = m = -2$. Then $\la = ((1^4),(1^{16}))$ is the bipartition linked with $\mu$ with $\ell(\la) = -3$ (see Figure \ref{paths before the proof}). The hyperplanes that we shall be interested in are $H_{-1/2}$, $H_{1/2}$ which are the hyperplanes of the fundamental alcove and the hyperplane $H_{m-1/2} = H_{-5/2}$ which is the left hyperplane of the alcove $\mathfrak{a}_{-1}$.

\begin{figure}[ht!]
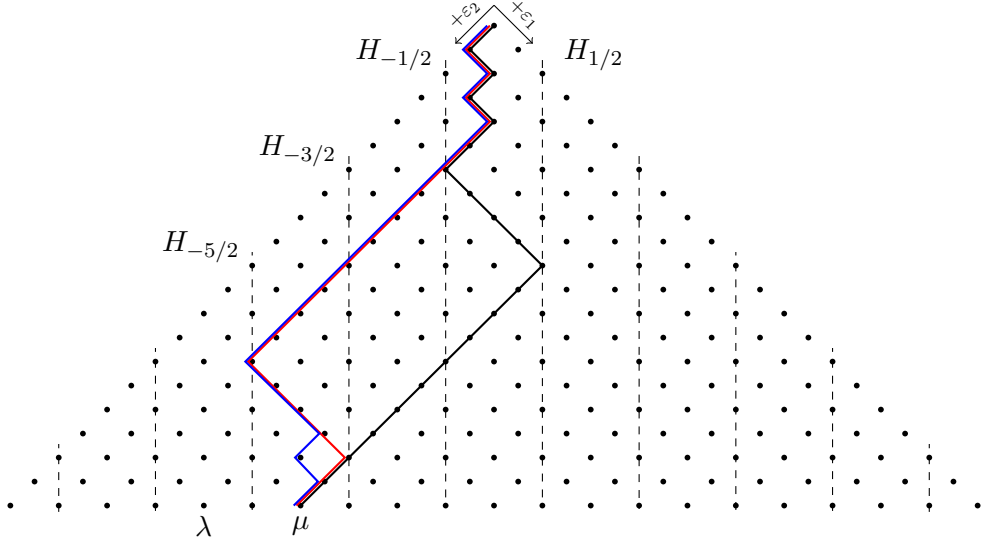

\[

\]
\caption{The blue and the black path label elements which belong in the image of $\varphi_{\la}^{\mu}$ whereas the red path labels an element not in the image of $\varphi_{\la}^{\mu}$.}
\label{paths before the proof}
\end{figure}
The black path is a path which intersects the hyperplane $H_{1/2}$ after intersecting the hyperplane $H_{-1/2}$. The blue path is a path last intersecting the hyperplane $H_{-5/2}$. Both paths belong to the image of the homomorphism $\varphi_{\la}^{\mu}$. On the other hand the red path does intersect the hyperplane $H_{-5/2}$, but it last intersects $H_{-3/2}$ and it does not belong to the image of $\varphi_{\la}^{\mu}$.
\end{eg}

Let $\mu = ((1^{\mu_1}),(1^{\mu_2}))\in\Bip(d)$ with $\ell(\mu) = m < 0$, i.e. $\mu_1 < \mu_2$. We shall construct a path $\TT\in\Path(\mu)$ which intersects hyperplanes $H_{-1/2}, \cdots,H_{m+3/2},H_{m-1/2}$ at exactly one point, hyperplane $H_{m+1/2}$ at exactly two points and it does not intersect hyperplane $H_{1/2}$, as follows. Let $j\in\mathbb{Z}$, $1\le j\le d-1$ be such that $\TT(i) = \TT^{\mu}(i)$, for any $1\le i\le j+1$, and recall from relation (\ref{step of path}) that $\TT(j) = c_{j,1}(\TT)\epsilon_1 + c_{j,2}(\TT)\epsilon_2$. Also let $n > j$ such that $\TT(n)\in H_{-1/2}$ and $\TT(n) = c_{n,1}(\TT)\epsilon_1 + c_{n,2}(\TT)\epsilon_2$ with $c_{n,1}(\TT) = c_{j,1}(\TT)$ and $c_{n,2}(\TT) = c_{j,2}(\TT) + n - j$. We denote by $a\in\mathbb{Z}$ the integer with the property $\TT(a)\in H_{m-1/2}$ and $c_{a,1}(\TT) = c_{j,1}(\TT)$, $c_{a,2}(\TT) = c_{n,2}(\TT) + |\ell(\mu)|e$. Finally, let $\TT(b)\in H_{m+1/2}$ be the second intersection point of $\TT$ with $H_{m+1/2}$. Note that the integers $j, n, a, b$ determine the path $\TT$. The diagram corresponding to the basis element $\psi_{\TT}$ is presented in Figure \ref{General diagram}.

\begin{figure}[ht]
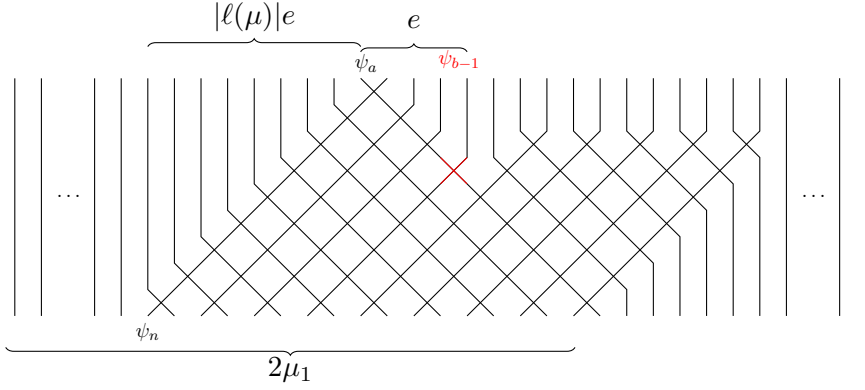

\[

\]
\caption{The general form of the diagram corresponding to the element $\psi_{\TT}$. Here the crossing marked with red is the crossing $\psi_{b-1}$ and the colouring has nothing to do with the residues adjacent to it.}
\label{General diagram}
\end{figure}

\begin{notation}
Let $\nu\in\Bip(d)$ and $\TT\in\mathsf{Path}(\nu)$ be a path. We denote by $\mathsf{t}_{n-1/2}^{\mathrm{last}}$ the last intersection point of the path $\TT$ with the hyperplane $H_{n-1/2}$, for some $n\in\mathbb{Z}$. Also we denote by $s_{n-1/2}^{\mathrm{last}}$ the reflection through that point with respect to the hyperplane $H_{n-1/2}$. 
\end{notation}

\begin{proof}[Proof of Proposition \ref{image of la}]
Same as in the proof of Proposition \ref{image of lapr} we cover the case $m\le 0$. Let $a\in\mathbb{Z}$, $1\le a < d$ be a fixed integer such that if $\alpha = ((1^{\alpha_{1}}),(1^{\alpha_{2}}))$ is a bipartition of $a$, then $\alpha_1 - \alpha_2\in H_{m-1/2}$. Also let $\TT^{\alpha}\in\Std(\alpha)$ be the $\alpha$-tableau highest in the lexicographic order. Consider the skew bipartition $\la\backslash\alpha$ and let $\SS, \SS^{\prime}\in\mathsf{Path}(\alpha\rightarrow\la)$ be length increasing paths which are highest and lowest in the lexicographic order respectively (see Figure \ref{Figure 8}). Recall that the elements of the set $\mathsf{Path}(\alpha\rightarrow\la)$ are paths starting from the bipartition $\alpha\in H_{m-1/2}$ and endpoint the bipartition $\la$.  All the remaining length increasing $\la\backslash\alpha$-paths placed between $\SS$ and $\SS^{\prime}$ can be obtained by multiplying with a product of transpositions on the tableau $\SS$ and we focus on $\SS$, $\SS^{\prime}$ for the ease of notation. We define the standard $\la$-tableaux $\TT := \TT^{\alpha}\circ\SS$ and $\TT^{\prime} := \TT^{\alpha}\circ \SS^{\prime}$ and let $\hat{\TT} := s_{m-1/2}\cdot\TT$ and $\hat{\TT}^{\prime} := s_{m-1/2}\cdot\TT^{\prime}$ be the reflection of those paths through the unique point they intersect the hyperplane $H_{m-1/2}$. Note that since the paths $\TT$, $\TT^{\prime}$ are length increasing paths, the basis elements $\psi_{\hat{\TT}}$, $\psi_{\hat{\TT}^{\prime}}$ corresponding to the paths $\hat{\TT}$, $\hat{\TT}^{\prime}$ belong the image of the homomorphism $\varphi_{\la}^{\mu}$. 

We shall prove that if the generators $\psi_{r}$, $a < r < d$ act on $\mathsf{\psi_{\hat{\TT}}}$ then $\psi_{r}\psi_{\hat{\TT}}$ is a non-zero element and it corresponds to a path which either last intersects $H_{m-1/2}$ or intersects $H_{1/2}$ after intersecting $H_{-1/2}$. Since $\psi_{r}\psi_{\hat{\TT}}$ belongs to the image of $\varphi_{\la}^{\mu}$, the new element will also belong to the image of $\varphi_{\la}^{\mu}$. For any $a < r < d$ such that $s_{r}\TT$ does not intersect $H_{m-1/2}$, $H_{m+1/2}$, it is straightforward that $\psi_{r}\psi_{\hat{\TT}} = \psi_{s_{r}\hat{\TT}}$ because $s_{r}\TT$ is the unique tableau with the desired residue sequence. Let $b\in\mathbb{Z}$, $a<b<d$, such that $(s_b\TT)(b)\in H_{m+1/2}$. Since $s_{b}\hat{\TT}\preceq \hat{\TT}$ we also have that  
\[
\psi_{b}\psi_{\hat{\TT}} = \psi_{s_{b}\hat{\TT}}
\] 
and the element $\psi_{b}\psi_{\hat{\TT}}$ is a non zero element which belongs to the image of the homomorphism $\varphi_{\la}^{\mu}$. We also need to prove that $\psi_{r}\psi_{\hat{\TT}^{\prime}}$, $1<r<d$ is a non zero element which belongs to the radical. Consider the element $\psi_{\hat{\TT}^{\prime}}$ and let $b\in\mathbb{Z}$ be such that $(s_{b}\TT^{\prime})(b)\in H_{m-3/2}$. This is the only interesting case as for the rest cases the result is straightforward. The transposition $s_{b}$ will appear in the reduced expression of $\hat{\TT}^{\prime}$ and it will be left exposed. Hence
\[
\psi_{b}\psi_{\hat{\TT}^{\prime}} = \psi_{b}^{2}\psi_{i_1}\cdots\hat{\psi_{b}}\cdots\psi_{i_k}\mathsf{e}(\ui^{\mu})
\]
with $\psi_{s_{b}\hat{\TT}^{\prime}} = \psi_{i_1}\cdots\hat{\psi_{b}}\cdots\psi_{i_k}\mathsf{e}(\ui^{\mu})$, where by $\hat{\psi_{b}}$ we mean that the generator $\psi_{b}$ does not appear in the product. Since $\mathsf{res}(({s_{b}\hat{\TT}^{\prime}})^{-1}(b)) = \mathsf{res}(({s_{b}\hat{\TT}^{\prime}})^{-1}(b+1)) + 1$, by applying the KLR relation (\ref{2.12}) we have that
\[
\psi_{b}\psi_{\hat{\TT}^{\prime}} = (y_{b+1}-y_{b})\psi_{s_{b}\hat{\TT}^{\prime}}.
\] 

\underline{{\bf Step 1}:} We shall prove that $y_{b+1}\psi_{s_{b}\hat{\TT}^{\prime}} = 0$. Let $(s_{b}\hat{\TT}^{\prime})(n)\in H_{-1/2}$, for some $n\in\mathbb{Z}$, be the unique intersection point of the path $s_{b}\hat{\TT}^{\prime}$ with the hyperplane $H_{-1/2}$. 

\begin{figure}[ht!]
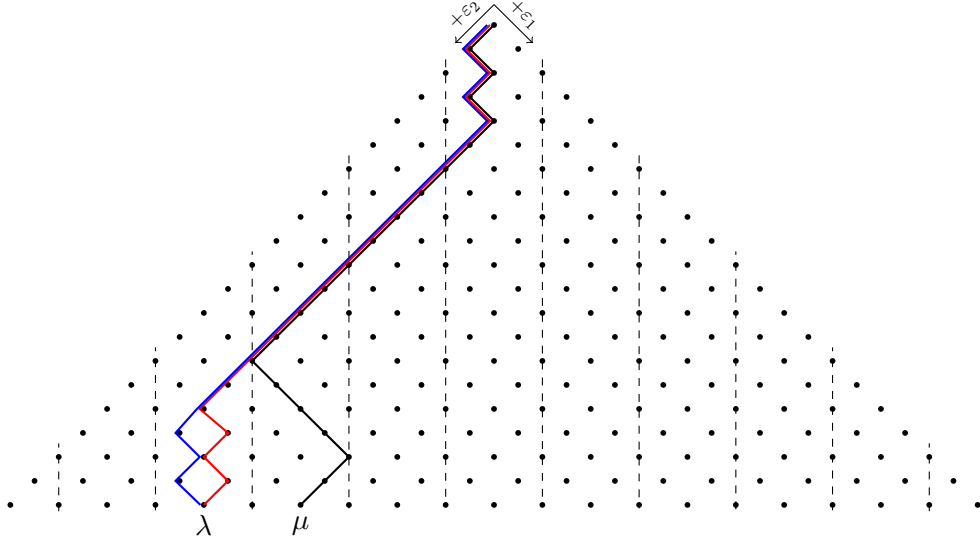

\[

\]
\vspace{-1.1cm}
\caption{Let $d = 20$, $e = 4$ and $\kappa = (0,2)$. For $a = 14$ and $b = 18$ the paths $\TT$ (red), $\TT^{\prime}$ (blue) and $s_{18}\hat{\TT}^{\prime}$ (black) are depicted above. The paths $\SS$ and $\SS^{\prime}$ are the bits of the $\TT$ and $\TT$ starting from the hyperplane $H_{-5/2}$ all the way down to $\la$.  In this case $n = 6$.}
\label{Figure 8}
\end{figure}

In order to compute the product $y_{b+1}\psi_{s_b\hat{\TT}}$ it is easier to consider the diagrammatic presentation of our algebra. In particular the diagram of the element $s_{b}\psi_{\hat{\TT}^{\prime}}$ is of the form of Figure \ref{General diagram}.


Note that the diagram consists of strands moving towards up to the right (UR-strands) and strands moving towards up to the left (UL-strands). If the $l^{\mathrm{th}}$ UR-strand (resp. UL-strand) carries the residue $i\in\mathbb{Z}/e\mathbb{Z}$ then the $(l+xe)^{\mathrm{th}}$, $x\in\mathbb{Z}_{>0}$, UR-strand (resp. UL-strand) also carries the residue $i\in\mathbb{Z}/e\mathbb{Z}$. We colour strands carrying the same residue with the same colour.  

We apply the generator $y_{b+1}$ on the element $\psi_{s_{b}\hat{\TT}^{\prime}}$ and we obtain the element corresponding to the following diagram.

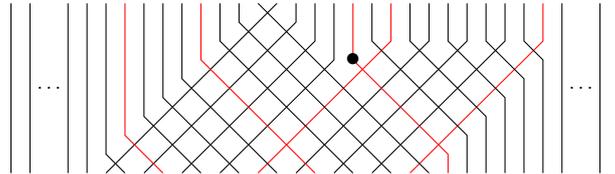
\[

\]

By applying the KLR-relation (\ref{2.10}) in the case that the residues coincide we get the following combination of diagrams:
\begin{equation} \label{first two summands}
%
\end{equation}  

We take the first summand of (\ref{first two summands}) and by reapplying the KLR relation (\ref{2.10}) we obtain two more summands.

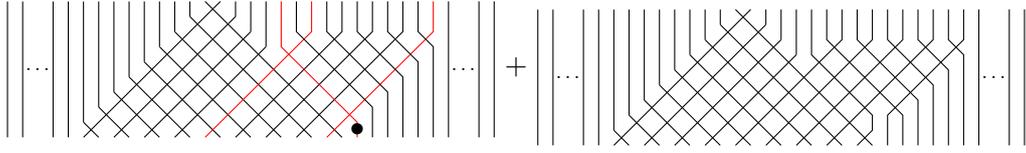
\[
%
\]

Those new summands are both equal to zero. The first one by the Garnir relation (\ref{2.23}) and the second one because it corresponds to a non standard tableau. We now consider the second summand of (\ref{first two summands}) and we apply the KLR relation (\ref{2.13}). We obtain the element
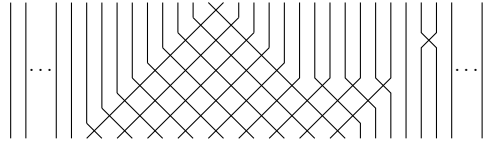
\[
%
\]
which is equal to zero because of the Garnir relations. Finally we get that $y_{b+1}\psi_{s_{b}\hat{\TT}^{\prime}} = 0$.   

\bigskip
\underline{{\bf Step 2}:}  Now we shall consider the product $y_{b}\psi_{s_{b}\hat{\TT}^{\prime}}$ and we shall distinguish between two cases according to the length of $\mu$. If $|\ell(\mu)| > 1$, then the unique element of $\mathsf{Path}(\mu,s_{b}\hat{\TT}^{\prime})$ with degree equal to $\mathrm{deg}(s_{b}\hat{\TT}^{\prime}) + 2$ is the path $\mathsf{V}_1 := s_{m+1/2}s_{m+3/2}s_{1/2}s_{-1/2}\cdot (s_{b}\hat{\TT}^{\prime})$, hence
\[
\psi_{b}\psi_{\hat{\TT}^{\prime}} = \alpha_{\mathsf{V}_1}\psi_{\mathsf{V}_1} 
\]   
with $\alpha_{\mathsf{V}_1}\in F$. If $|\ell(\mu)| = 1$, then the unique path in $\mathsf{Path}(\mu,s_{b}\hat{\TT}^{\prime})$ with the desired property is the path $\mathsf{V}_2 := s_{1/2}^{2}s_{1/2}^{1}\cdot(s_{b}\hat{\TT}^{\prime})$, hence 
\[
\psi_{b}\psi_{s_{b}\hat{\TT}^{\prime}} = \alpha_{\mathsf{V}_2}\psi_{\mathsf{V}_2} 
\] 
with $\alpha_{\mathsf{V}_2}\in F$. In order to prove that the homomorphism $\varphi_{\la}^{\mu}$ is injective it suffices to prove that the scalars $\alpha_{\mathsf{V}_1}$, $\alpha_{\mathsf{V}_2}\in F$ are non-zero. We prove it for the scalar $\alpha_{\mathsf{V}_1}\in F$ since the proof for $\alpha_{\mathsf{V}_2}\in F$ will be a subcase. The element $y_{b}\psi_{s_b\hat{\TT}^{\prime}}$ corresponds to the diagram 
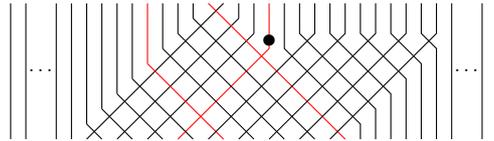
\[

\]  

By the KLR relation (\ref{2.11}) we have that the above element is equal to the following combination of diagrams:

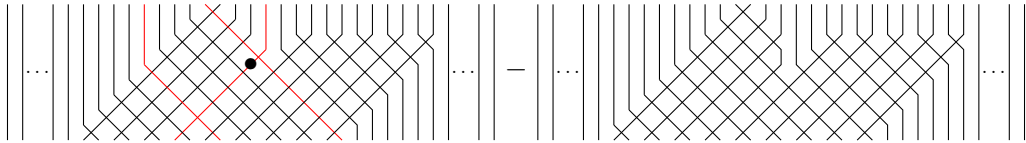
\[
%
\]

By using similar arguments as above, the first summand is zero. We now consider the second summand and we apply the KLR-relation (\ref{2.13}), hence we obtain the diagram

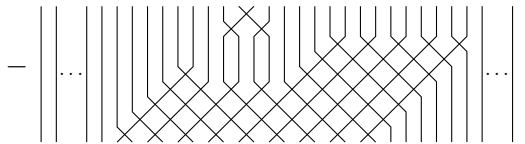
\[
%
\]

We apply the KLR-relation (\ref{2.12}) for the case that the residues are not equal and the do not differ by one. The we obtain the diagram
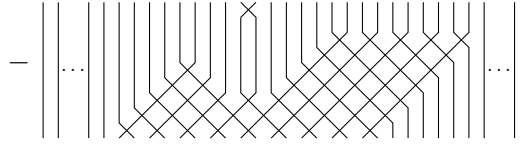
\[
%
\]
in which we can apply the KLR relation (\ref{2.12}) for the case that the residues differ by one. Then we get the following sum of diagrams.
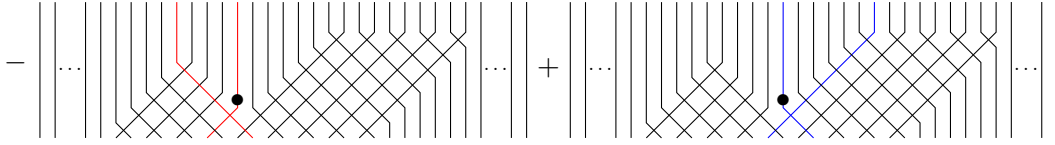
\[
%
\]
where strands with different colours carry different residues which differ by one. We apply the KLR relations (\ref{2.10}), (\ref{2.11}) and (\ref{2.13}) appropriately until we obtain reduced diagrams. Then the only non-zero summand is of the form
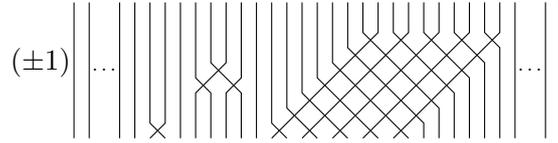
\[
%
\]

Hence we have proven that the scalar $\alpha_{\mathsf{V}_1}\in F$ is equal to $\pm 1$ and we shall not be interested in keeping track of its value. As a result the homomorphism $\varphi_{\la}^{\mu}$ is injective homomorphism. 
 
In any case the element $\psi_{b}\psi_{\hat{\TT}^{\prime}}$ corresponds to the path $\mathsf{V}_1 := s_{m+1/2}s_{m+3/2}s_{1/2}s_{-1/2}\cdot (s_{b}\hat{\TT}^{\prime})$ which intersects the hyperplane $H_{1/2}$ after intersecting the hyperplane $H_{-1/2}$. By repeating the same procedure for all admissible integers $a\in\mathbb{Z}$ we prove that the paths which either last intersect $H_{m-1/2}$ or intersect $H_{1/2}$ after intersecting $H_{-1/2}$, correspond to elements in the image of $\varphi_{\la}^{\mu}$.

\bigskip
{\bf \underline{A bijection:}} In order to complete the proof we need to prove that any element in the image of $\varphi_{\la}^{\mu}$ either last intersects $H_{m-1/2}$ or intersects $H_{1/2}$ after intersecting $H_{-1/2}$. For that purpose it suffices to show that the map 
\[
\Phi\colon\mathsf{Path}(\la)\lxr\mathsf{Path}(\mu)
\] 
defined by 
\[
\Phi(\mathsf{U}) := 
\begin{cases}
s_{m-1/2}^{\mathrm{last}}\cdot\mathsf{U}, \hspace{1.7cm} \text{if $\mathsf{U}$ last intersects $H_{m-1/2}$} \\ 
s_{m-1/2}^{\mathrm{last}}s_{1/2}^{\mathrm{last}}s_{-1/2}^{\mathrm{last}}\cdot\mathsf{U}, \ \text{otherwise}
\end{cases}
\]
is an injective map of degree one, with image containing the paths in $\mathsf{Path}(\mu)$ which either last intersect $H_{m-1/2}$ or intersect $H_{1/2}$ after intersecting $H_{-1/2}$. Let $\mathsf{U}\in\mathsf{Path}(\mu)$ be a path which last intersects the hyperplane $H_{m-1/2}$ at the point $\mathsf{u}_{m-1/2}^{\mathrm{last}}$. Then we have that
\[
\Phi(s_{m-1/2}^{\mathrm{last}}\cdot\mathsf{U}) = \mathsf{U}
\]
hence $\mathsf{U}$ belongs to the image of the map $\Phi$. Consider an arbitrary path $\mathsf{V}\in\mathsf{Path}(\mu)$ which intersects both hyperplanes $H_{-1/2}$ and $H_{1/2}$. Suppose that if $\mathsf{v}_{1/2}^{\mathrm{last}} = \mathsf{V}(n_2)$ is the last intersection point with the hyperplane $H_{1/2}$, then there exists an intersection point $\mathsf{v}_{-1/2}^{l} = \mathsf{V}(n_1)$ with $n_1 < n_2$ and assume that $n_1$ is the greatest integer with that property. Moreover let $\mathsf{v}_{m+1/2}^{\mathrm{last}} = \mathsf{V}(n_3)$ be the last intersection point of $\mathsf{V}$ with $H_{m+1/2}$. Then 
\[
\Phi(s_{-1/2}^{l}s_{1/2}^{\mathrm{last}}s_{m+1/2}^{\mathrm{last}}\cdot\mathsf{V}) = \mathsf{V}
\]    
hence $\mathsf{V}$ belongs to the image of $\Phi$. Since both those types of paths belong to the image of $\varphi_{\la}^{\mu}$ we have proven that any element in the image corresponds to a path of that form. The fact that $\Phi$ is of degree $1$ is straightforward by its construction. \qedhere
\end{proof}

\section{Bases of simple modules}

In this section we assume that $F$ is a field of characteristic $0$ and we shall construct the bases of simple modules for the algebra $\blob$. Recall from Section 3 that for a given bipartition $\mu\in\Bip(d)$ with $\ell(\mu) \le 0$ we fix two bipartitions $\la$, $\lapr$ and consider the homomorphisms $\varphi_{\la}^{\mu}$, $\varphi_{\lapr}^{\mu}$ of Definition \ref{homomorphisms}. Note that everything works on the same way if $\ell(\mu) > 0$, so we restrict ourselves to the previous case. Let us denote by $\Image \ \varphi_{\lapr}^{\mu}$ and $\Image \ \varphi_{\la}^{\mu}$  the images of the above homomorphisms, constructed in Propositions \ref{image of lapr} and \ref{image of la} respectively. We denote by $E(\mu)$ the quotient module
\[
E(\mu) := \Delta(\mu)/(\Image \ \varphi_{\lapr}^{\mu}+\Image \ \varphi_{\la}^{\mu})
\]
of $\Delta(\mu)$ modulo the sum of the images of the homomorphisms. From the results of the previous section we have that when $\mu$ belongs to an alcove, $E(\mu)$ is spanned by elements corresponding to paths which do not intersect the hyperplane $H_{1/2}$ and they do not last intersect the hyperplane $H_{m-1/2}$.  In the hyperplane case we have that the module $E(\mu)$ is spanned by elements $\psi_{\TT}$ where $\TT$ is a path which does not intersect the hyperplane $H_{1/2}$.

For any path $\TT\in\Path(\mu)$ with $\psi_{\TT}\in E(\mu)$ we shall construct a path $\bar{\TT}\in\Path(\mu)$ with $\psi_{\bar{\TT}}\in E(\mu)$ and $\deg(\bar{\TT}) = - \deg(\TT)$. We denote by $\mathsf{t}_{n-1/2}^{1}, \mathsf{t}_{n-1/2}^{2},\cdots$ the intersection points of $\TT$ with the hyperplane $H_{n-1/2}$ for some $n\le 0$. For the construction of $\bar{\TT}$ we focus our attention on the intersection points of $\TT$ with the hyperplanes. Let $\mathsf{T}^{i}_{n-1/2}$ be an intersection point of $\TT$ with the hyperplane $H_{n-1/2}$. If the next point that $\TT$ intersects any hyperplane is the point $\mathsf{t}_{n-1/2}^{i+1}$ then for all point between $\mathsf{T}^{i}_{n-1/2}$ and $\mathsf{T}^{i+1}_{n-1/2}$ (which belong to an alcove) we have that $\bar{\TT}(a) := (s_{n-1/2}^{i}\cdot\TT)(a)$. We need to consider the case that the next intersection point is $\mathsf{t}_{n-3/2}^{j}$, for some $j$, or $\mathsf{t}_{n-1/2}^{i}$ is the last intersection point of $\TT$ with any hyperplane. In these cases, for the points between $\mathsf{t}_{n-1/2}^{i}$ and $\mathsf{t}_{n-3/2}^{j}$ or the points from $\mathsf{t}_{n-1/2}^{i}$ until the end of the path respectively, we have that $\bar{\TT}(a) := \TT(a)$.

Note that the above construction does not depend on whether the bipartition $\mu$ lies in an alcove or on a hyperplane.

\begin{eg}
Suppose that $d = 24$, $e = 4$ and $\kappa = (0,2)$. We consider the bipartition $\mu = ((1^8),(1^{16}))\in\Bip(24)$ and let $\TT\in\mathsf{Path}(\mu)$ be the black path in Figure \ref{bar inv picture}  which corresponds to the basis element $\psi_{\TT}\in E(((1^8),(1^{16})))$. The path $\TT$ has degree $\deg(\TT) = -2$.

\begin{figure}
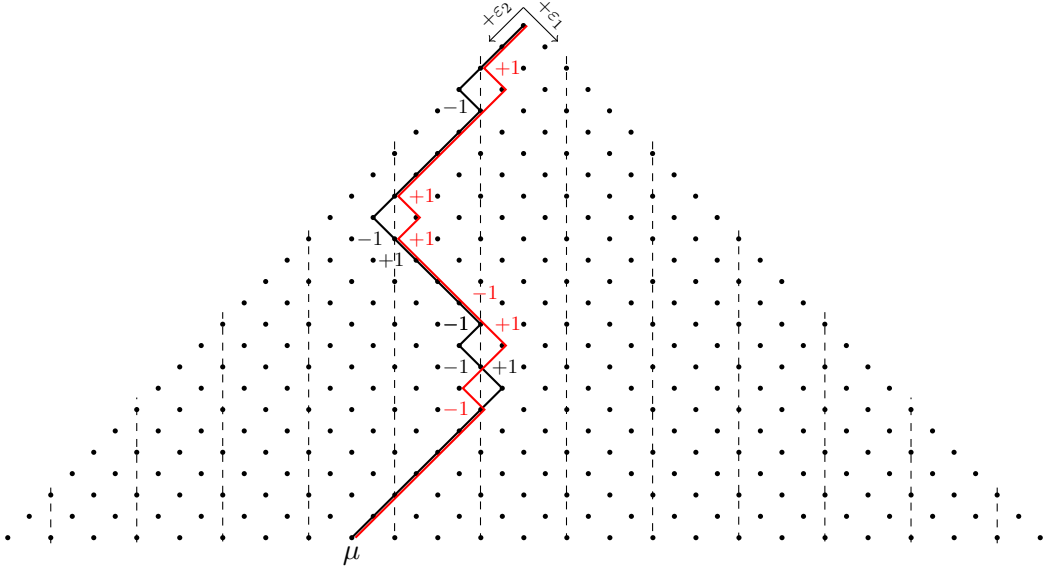

\[

\]
\caption{The black path is the path $\TT$ and the red path is the path $\bar{\TT}$. The numbers in black and red are the integers contributing to the degree of $\TT$ and $\bar{\TT}$ respectively.}
\label{bar inv picture}
\end{figure}

The path $\bar{\TT}\in\mathsf{Path}(\mu)$ obtained by the procedure we described before, is the red path in Figure \ref{bar inv picture}. One can readily check that $\deg(\bar{\TT}) = 2 = - \deg(\TT)$.

\end{eg}

\begin{rmk}
Suppose that $\mu\in\Bip(d)$ with $\ell(\mu) = m > 0$. Then the elements spanning the module $E(\mu)$ are of the form $\psi_{\TT}$ where $\TT$ is a path which does not intersect $H_{-1/2}$ and does not last intersect the hyperplane $H_{m+1/2}$. 
\end{rmk}

Using the notions we defined above we can state and prove the following theorems. Those theorems are two of the main results of our paper and gives a precise description of the basis of an irreducible representation of the blob algebra over a field of characteristic 0, in the alcove and hyperplane cases.

\begin{thm}\label{bases of simple modules}
Let $\mu\in\Bip(d)$ with $\ell(\mu) = m$. The module $E(\mu)$ is equal to the simple head $L(\mu)$, hence
\begin{enumerate}
\item if $m\le 0$
\[
L(\mu) = \mathrm{span}_{F}\left\{\psi_{\TT} \ |{\scalefont{1.3} \substack{ \ \TT\in\mathsf{Path}(\mu), \ \TT \ \text{does not intersect} \ H_{1/2} \\ \text{and} \ \text{does not last intersect} \ H_{m-1/2}}}\right\},
\]
\item if $m > 0$
\[
L(\mu) = \mathrm{span}_{F}\left\{\psi_{\TT} \ |{\scalefont{1.3} \substack{ \ \TT\in\mathsf{Path}(\mu), \ \TT \ \text{does not intersect} \ H_{-1/2} \\ \text{and} \ \text{does not last intersect} \ H_{m+1/2}}}\right\}.
\]
\end{enumerate}
\end{thm}

\begin{proof}
By Proposition \ref{simples bar invariant} and Theorem \ref{gr dec numbers} (the latter of which holds because we are working over   a field of characteristic zero) we can characterise  $L(\mu)$ as  the largest vector-space
 quotient of $\Delta(\mu)$ which has a   bar-invariant character (i.e. fixed under interchanging $t$ and $t^{-1}$) and carries  a $\blob$-module structure. 
For any path $\TT\in\mathsf{Path}(\mu)$ with $\psi_{\TT}\in E(\mu)$ we have already shown that there exists a path $\bar{\TT}$ with $\deg(\bar{\TT}) = - \deg(\TT)$ and therefore  $\dim_{t}(E(\mu))$ is bar-invariant. 
Finally, the $\blob$-submodule by which we have quotiented (to obtain $E(\mu)$) was
 constructed as the image of homomorphisms of strictly positive degree (degree 1) and therefore all its simple constituents appear with {\em strictly} positive degree shift (by application of Theorem \ref{gr dec numbers} to $\Delta(\la)$ and $\Delta(\la')$).  The result follows.   \qedhere

\end{proof}

The following theorem is the analogous of Theorem \ref{bases of simple modules} in the hyperplane case.

\begin{thm}
Let $\mu\in\Bip(d)$ with $\ell(\mu) = m - 1/2$. The module $E(\mu)$ is equal to the simple head $L(\mu)$, hence 
\begin{enumerate}
\item if $m\le 0$
\[
L(\mu) = \mathrm{span}_{F}\{\psi_{\TT} \ | \ \TT\in\mathsf{Path}(\mu), \ \TT \ \text{does not intersect} \ H_{1/2}\},
\]
\item if $m > 0$
\[
L(\mu) = \mathrm{span}_{F}\{\psi_{\TT} \ | \ \TT\in\mathsf{Path}(\mu), \ \TT \ \text{does not intersect} \ H_{-1/2}\}.
\]
\end{enumerate}
\end{thm}
\begin{proof}
The proof is identical to the proof of Theorem \ref{bases of simple modules}.
\end{proof}

\section{BGG resolutions of simple representations}

\subsection{Composition of one-column homomorphisms}
In this section we shall compute the composition of certain one-column homomorphisms. We consider two bipartitions $\alpha,\gamma\in\Bip(d)$ such that $|\ell(\alpha)| = |\ell(\gamma)| + 2$ and without loss of generality we may assume that $\ell(\gamma) < 0$. Then we can either have $\ell(\alpha) < 0$ or $\ell(\alpha) > 0$ and let $\beta, \beta^{\prime}\in\Bip(d)$ be the bipartitions with $|\ell(\beta)| = |\ell(\beta^{\prime})| = |\ell(\gamma)| + 1$ for which we have constructed the homomorphisms $\varphi_{\beta}^{\gamma},\varphi_{\beta^{\prime}}^{\gamma}$ of Section \ref{Section 3}. In a case as above we can consider the following ``diamond" diagram:

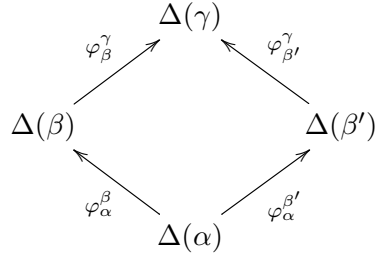
\[
\xymatrix{&\Delta(\gamma) \\
\Delta(\beta)\ar[ur]^-{\varphi_{\beta}^{\gamma}} && \Delta(\beta^{\prime})\ar[ul]_-{\varphi_{\beta^{\prime}}^{\gamma}} \\
&\Delta(\alpha)\ar[ul]^-{\varphi_{\alpha}^{\beta}} \ar[ur]_-{\varphi_{\alpha}^{\beta^{\prime}}}
}
\]
The aim of this section is to compute the compositions of the homomorphisms in such diamonds and prove that those are commutative or anti-commutative.

In the next proposition we shall assume that $\ell(\alpha) < 0$, as everything works similarly when $\ell(\alpha) > 0$.

\begin{prop}\label{comp of one column homs}
Let $\alpha,\gamma\in\Bip(d)$ with $|\ell(\alpha)| = |\ell(\gamma)| + 2$. Then 
\[
(\varphi_{\beta}^{\gamma}\circ\varphi^{\beta}_{\alpha})(\psi_{\TT^{\alpha}}) = (-1)^{|\ell(\gamma)|}\psi_{s_{\ell(\gamma) + 1/2}s_{1/2}s_{-1/2}s_{\ell(\beta) - 1/2}\cdot\TT^{\alpha}}
\]
and
\[
(\varphi_{\beta^{\prime}}^{\gamma}\circ\varphi_{\alpha}^{{\beta}^{\prime}})(\psi_{\TT^{\alpha}}) = \psi_{s_{1/2}s_{-1/2}\cdot\TT^{\alpha}}.
\]
In particular the diamond will either be commutative or anti-commutative, depending on the number $|\ell(\gamma)|$.
\end{prop}

\begin{proof}
Let $\alpha = ((1^{\alpha_1}),(1^{\alpha_2}))$, $\beta = ((1^{\beta_1}),(1^{\beta_2}))$ and $\gamma = ((1^{\gamma_1}),(1^{\gamma_2}))$. The composition $\Delta(\alpha)\lxr\Delta(\beta)\lxr\Delta(\gamma)$ is harder to compute than the composition $\Delta(\alpha)\lxr\Delta(\beta^{\prime})\lxr\Delta(\gamma)$ and we shall start by computing it. Consider the generator $\psi_{\TT^{\alpha}}$ of the cell module $\Delta(\alpha)$. Then
\begin{eqnarray*}
(\varphi_{\beta}^{\gamma}\circ\varphi^{\beta}_{\alpha})(\psi_{\TT^{\alpha}}) & = & \varphi_{\beta}^{\gamma}(\varphi_{\alpha}^{\beta}(\psi_{\TT^{\alpha}})) \\
& = & \varphi_{\beta}^{\gamma}(\psi_{s_{\ell(\beta)-1/2}\cdot\TT^\alpha}). 
\end{eqnarray*}
Let $\TT := s_{\ell(\beta)-1/2}\cdot\TT^\alpha\in\mathsf{Path}(\beta)$ and $\TT = w_{\TT}\TT^{\beta}$. Then we have that
\begin{eqnarray*}
\varphi_{\beta}^{\gamma}(\psi_{s_{\ell(\beta)-1/2}\cdot\TT^\alpha}) & = & \varphi_{\beta}^{\gamma}(\psi_{w_{\TT}}\psi_{\TT^{\beta}}) \\ 
& = & \psi_{w_{\TT}}\varphi_{\beta}^{\gamma}(\psi_{\TT^{\beta}}) \\
& = & \psi_{w_{\TT}}\psi_{s_{\ell(\beta)+1/2}\cdot\TT^{\beta}}.
\end{eqnarray*}
The element $\psi_{\TT}\in\Delta(\beta)$ corresponds to a diagram of the form
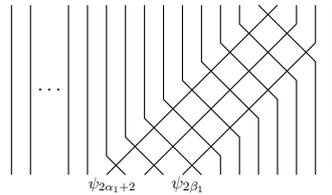
\[

\] 
while the element $\psi_{s_{\ell(\beta)+1/2}\cdot\TT^{\beta}}\in \Delta(\gamma)$ corresponds to a diagram of the form
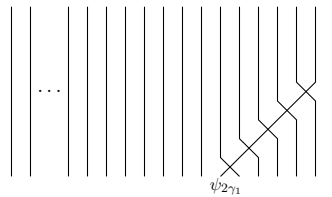
\[
%
\]
For the multiplication $\psi_{w_{\TT}}\psi_{s_{\ell(\beta)+1/2}\cdot\TT^{\beta}}$ we concatenate the above diagrams. Hence we obtain the diagram
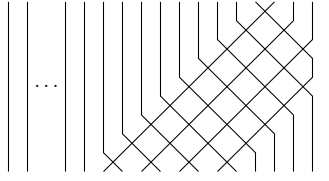
\[
%
\]
and by applying the KLR-relation (\ref{2.13}) in the case the middle residue is less by one that the adjacent residues we get the diagram
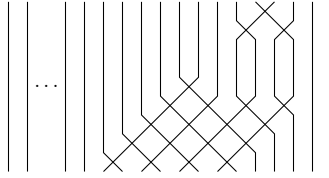
\[
%
\]
Then we apply the KLR relation (\ref{2.12}) in the case that the residues are not equal and they do not differ by one and we get the diagram 
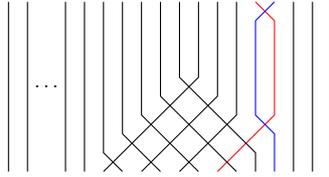
\[
%
\]

Since the strands we have marked in red and blue carry the different but adjacent residues, we apply the KLR-relation (\ref{2.12}) and we reduce our computation to a computation similar to the one in the proof of Proposition \ref{image of la} where we compute the effect of $y$-generators. Hence, we have that 
\begin{eqnarray*}
(\varphi_{\beta}^{\gamma}\circ\varphi^{\beta}_{\alpha})(\psi_{\TT^{\alpha}}) & = & \psi_{w_{\TT}}\psi_{s_{\ell(\beta)+1/2}\cdot\TT^{\beta}} \\
& = & (-1)^{|\ell(\gamma)|}\psi_{s_{\ell(\gamma) + 1/2}s_{1/2}s_{-1/2}s_{\ell(\beta) - 1/2}\cdot\TT^{\alpha}}.
\end{eqnarray*}

Now we shall compute the composition $\Delta(\alpha)\lxr\Delta(\beta^{\prime})\lxr\Delta(\gamma)$ on the generator $\psi_{\TT^{\alpha}}$ of the cell module $\Delta(\alpha)$. We have that 
\begin{eqnarray*}
(\varphi_{\beta^{\prime}}^{\gamma}\circ\varphi_{\alpha}^{{\beta}^{\prime}})(\psi_{\TT^{\alpha}}) & = & \varphi_{\beta^{\prime}}^{\gamma}(\varphi_{\alpha}^{\beta^{\prime}}(\psi_{\TT^{\alpha}})) \\ 
& = & \varphi_{\beta^{\prime}}^{\gamma}(\psi_{s_{-1/2}\cdot\TT^{\alpha}})
\end{eqnarray*}

Let $\SS := s_{-1/2}\cdot\TT^{\alpha}\in\mathsf{Path}(\beta^{\prime})$ and $\SS = w_{\SS}\TT^{\beta^{\prime}}$. Then we have that 
$$
\varphi_{\beta^{\prime}}^{\gamma}(\psi_{s_{-1/2}\cdot\TT^{\alpha}}) = \varphi_{\beta^{\prime}}^{\gamma}(\psi_{\SS}) = \varphi_{\beta^{\prime}}^{\gamma}(\psi_{w_{\SS}}\psi_{\TT^{\beta^{\prime}}}) = \psi_{w_{\SS}}\varphi_{\beta^{\prime}}^{\gamma}(\psi_{\TT^{\beta^{\prime}}}).
$$

By Definition \ref{homomorphisms} we have that $\varphi_{\beta^{\prime}}^{\gamma}(\psi_{\TT^{\beta^{\prime}}}) = \psi_{s_{1/2}\cdot\TT^{\beta^{\prime}}}$ and let $\mathsf{U} := s_{1/2}\cdot\TT^{\beta^{\prime}} \in\mathsf{Path}(\gamma)$. Then

$$
\varphi_{\beta^{\prime}}^{\gamma}(\psi_{s_{-1/2}\cdot\TT^{\alpha}}) = \psi_{w_{\SS}}\psi_{s_{1/2}\cdot\TT^{\beta^{\prime}}} = \psi_{w_{\SS}}\psi_{w_{\mathsf{U}}}\psi_{\TT^{\gamma}}.
$$

The final equality gives the desired result because the product of generators $\psi_{w_{\SS}}\psi_{w_{\mathsf{U}}}$ corresponds to reduced transposition, hence $\psi_{w_{\SS}}\psi_{w_{\mathsf{U}}}\psi_{\TT^{\gamma}}\in\mathsf{Path}(\gamma)$ is equal to the element $\psi_{s_{-1/2}s_{1/2}\cdot\TT^{\alpha}}$ as required. \qedhere
\end{proof}

\subsection{BGG resolutions for the alcove and hyperplane case}
Let $F$ be a field of characteristic zero. In this section we attach to any bipartition $\la\in\Bip(d)$ a complex $C_{\bullet}(\la)$ called the BGG resolution for the irreducible representation $L(\la)$.

In the case that the simple representation is indexed by a bipartition $\la\in\Bip(d)$ with $\la\in H_{n-1/2}$, $n\in\mathbb{Z}$, the BGG resolution has an easy form. In the next proposition we construct a BGG resolution for the irreducible representation $L(\la)$, $\la\in H_{n-1/2}$, for some $n\in\mathbb{Z}$.

\begin{prop} \label{BGG resolution for the wall case}
Let $\la\in\Bip(d)$ with $\la\in H_{n-1/2}$, for some $n\in\mathbb{Z}$. We have the short exact sequence 
\[
\xymatrix{C_{\bullet}(\la)\colon & 0 \ar[r] & \Delta(\mu)\langle |\ell(\mu)| - |\ell(\la)| \rangle \ar[r]^-{\varphi_{\mu}^{\la}} & \Delta(\la) \ar[r] & L(\la)\ar[r] & 0
}
\]
where $\mu\in\Bip(d)$ with $|\ell(\mu)| = |\ell(\la)| + 1$ and $\ell(\mu) = -( \ell(\la) + 1)$, with
\[
H_{i}(C_{\bullet}(\la)) = 
\begin{cases}
L(\la), \ \text{if} \ i = 0 \\
0, \hspace{0.7cm} \text{otherwise}.
\end{cases}
\]
\end{prop}

\begin{proof}
The result is straightforward by using the fact that $\mathrm{Coker}(\varphi_{\mu}^{\la}) = L(\la)$. \qedhere
\end{proof}


Now we shall construct BGG resolutions for the simple modules indexed by bipartitions which belong to an alcove. Let $\la\in\Bip(d)$ be a bipartition such that $\la\in\mathfrak{a}_n$, $n\in\mathbb{Z}$ and let us denote by $\nu_{i}$, $\nu_{i}^{\prime}$ the bipartitions -in the same linkage class as $\la$- such that $|\ell(\nu_{i})| = |\ell(\nu_{i}^{\prime})| = |\ell(\la)| + i$. We set 
\[
C_{\bullet}(\la) := (C_{i}(\la))_{i\ge 0}
\]
where
\[
C_{0}(\la) := \Delta(\la)
\]
and
\begin{equation*}
C_{i}(\la) := \bigoplus_{\nu = \nu_{i},\nu_{i}^{\prime}}\Delta(\nu)\langle |\ell(\nu)| - |\ell(\la)| \rangle
\end{equation*}
for $i>0$. We define the maps 
\[
\delta_{i}\colon C_{i+1}(\la)\lxr C_{i}(\la)
\]
between those components. For $i = 0$ we have that 
\begin{equation}\label{delta0}
\delta_{0} := 
\begin{pmatrix}
\varphi_{\nu_1}^{\la} & \varphi_{\nu_{1}^{\prime}}^{\la} 
\end{pmatrix}.
\end{equation}
For $i > 0$ we shall distinguish between two cases on the number $|\ell(\la)| + i$. In particular if $|\ell(\la)| + i$ is even, we set
\begin{equation}\label{deltai}
\delta_{i} :=
\begin{pmatrix}
-\varphi_{\nu_{i+1}}^{\nu_{i}} & \varphi_{\nu_{i+1}^{\prime}}^{\nu_{i}} \\
\varphi_{\nu_{i+1}}^{\nu_{i}^{\prime}} & -\varphi_{\nu_{i+1}^{\prime}}^{\nu_{i}^{\prime}}
\end{pmatrix}
\end{equation}
whereas if it is odd, we set
\begin{equation}\label{deltai'}
\delta_{i} :=
\begin{pmatrix}
\varphi_{\nu_{i+1}}^{\nu_{i}} & -\varphi_{\nu_{i+1}^{\prime}}^{\nu_{i}} \\
-\varphi_{\nu_{i+1}}^{\nu_{i}^{\prime}} & \varphi_{\nu_{i+1}^{\prime}}^{\nu_{i}^{\prime}}
\end{pmatrix}.
\end{equation}
Note that there is the possibility that not both the rightmost and leftmost alcove contain bipartitions linked with $\la$. In that case let $\nu_1,\nu_1^{\prime},\cdots,\nu_{k}\in\Bip(d)$ be the bipartitions linked with $\la$. Then we define the maps $\delta_{i}\colon C_{i+1}(\la)\lxr C_{i}(\la)$ are defined exactly like the maps (\ref{delta0}), (\ref{deltai}) and (\ref{deltai'}) for $0\le i \le k-1$. For $i = k$ we define
\begin{equation}
\delta_{k} := 
\begin{pmatrix}
\varphi_{\nu_{k}}^{\nu_{k-1}} \\
\varphi_{\nu_{k}}^{\nu_{k-1}^{\prime}}
\end{pmatrix}.
\end{equation}

\begin{prop}
Let $\la\in\Bip(d)$ be a bipartition such that $\la\in\mathfrak{a}_n$, $n\in\mathbb{Z}$. For the pair $(C_{\bullet}(\la),(\delta_{i})_{i\ge 0})$ we have that
\[
\mathrm{Im}(\delta_{i+1})\subset \mathrm{Ker}(\delta_{i}).
\]
for any $i\ge 0$, in other words the pair $(C_{\bullet}(\la),(\delta_{i})_{i\in\mathbb{Z}})$ is a (chain) complex. 
\end{prop}
\begin{proof}
The result is straightforward from Proposition \ref{comp of one column homs}. 
\end{proof}

Recall that $I = \mathbb{Z}/e\mathbb{Z}$ and let $r\in I$ be a given residue. The \textsf{$r$-restriction functor} 
\[
r-\mathrm{res}_{d-1}^{d}\colon \mathrm{mod}-\blob\lxr\mathrm{mod}-\mathsf{B}_{d-1}^{\kappa}
\]
is defined by 
\[
M\longmapsto \sum_{\ui = (i_1,i_2,\cdots,i_{d-1},r)\in I^{d-1}\times\{r\}}\mathsf{e}(\ui)M
\]
and we have that 
\[
\mathrm{res}_{d-1}^{d} = \sum_{r\in I}r-\mathrm{res}_{d-1}^{d}.
\]

\begin{rmk}\label{remark about res functor}
Suppose that $\la\in\Bip(d)$ and $\la\in\mathfrak{a}_n$, $n\in\mathbb{Z}$. If $r\in I$ then we have that $\la$ has either $0$ or $1$ removable $r$-nodes. We shall denote by $\mathsf{E}_r(\la)$ the unique bipartition which differs from $\la$ by removing an $r$-node. Consider the cell module $\Delta_{d}(\la)\in\mathrm{mod}-\blob$. We have that 
\[
r-\mathrm{res}_{d-1}^{d}(\Delta_{d}(\la)) = 
\begin{cases}
\Delta_{d-1}(\mathsf{E}_{r}(\la)), \ \text{if $\mathsf{Rem}_{r}(\la)\neq \emptyset$} \\
0, \hspace{2cm} \text{otherwise}
\end{cases}
\]
where $\Delta_{d-1}(\mathsf{E}_{r}(\la))$ is a cell module in $\mathrm{mod}-\mathsf{B}_{d-1}^{\kappa}$. 
\end{rmk}

\begin{defin}
Let $\la\in\Bip(d)$. The complex
\[
\xymatrix{0 \ar[r] & C_{\bullet}(\la)  \ar[r] & L(\la) \ar[r] & 0
}
\]
is called \textsf{BGG resolution} of $L(\la)$ if 
\[
H_{i}(C_{\bullet}(\la)) = 
\begin{cases}
L(\la), \ \text{if} \ i = 0 \\
0, \hspace{0.7cm} \text{otherwise}.
\end{cases}
\]
\end{defin}

\begin{thm}
Let $\la\in\Bip(d)$ be a bipartition such that $\la\in\mathfrak{a}_n$, $n\in\mathbb{Z}$ and   
\[
C_{\bullet}(\la) := \bigoplus_{\nu\trianglelefteq\la} \Delta(\nu)\langle |\ell(\nu)| - |\ell(\la)| \rangle.
\] 
The $\blob$-complex
\[
\xymatrix{0 \ar[r] & C_{\bullet}(\la)  \ar[r] & L(\la) \ar[r] & 0
}
\]
with differentials $\delta_{i}\colon C_{i+1}(\la)\lxr C_{i}(\la)$ the maps defined above is a BGG resolution for the simple representation $L(\la)$. Moreover we have that 
\[
\mathrm{res}_{d-1}^{d}(L_{d}(\la)) = \bigoplus_{\square\in\mathsf{Rem}(\la)}L_{d-1}(\la-\square).
\]
\end{thm} 

\begin{proof}
Let $\la\in\Bip(d)$ with $\la\in\mathfrak{a}_n$, for some $n\le 0$. Note that everything works analogously when $n > 0$. In order to prove that our $\blob$-complex is a BGG resolution for the simple representation $L(\la)$ we need to show that
\[
H_{i}(C_{\bullet}(\la)) = 
\begin{cases}
L_{d}(\la), \ \text{if} \ i = 0 \\
0, \hspace{0.8cm} \text{otherwise}
\end{cases}.
\]
Recall that BGG resolutions and bases for the hyperplane case are already done. We assume, by induction, that the theorem holds for any bipartition $\la\in\Bip(d-1)$ where $\la$ belongs to an alcove.  We also have that 
\[
\mathrm{res}_{d-1}^{d}(C_{\bullet}(\la)) = \bigoplus_{r\in I}C_{\bullet}(\mathsf{E}_{r}(\la)).
\]
We shall consider one residue at a time. The bipartition $\la$ belongs to an alcove, hence as we mentioned in Remark \ref{remark about res functor} there will be either $0$ or $1$ removable $r$-nodes. For each residue we have 3 different cases.
\begin{itemize}
\item Suppose that the bipartition $\mathsf{E}_{r}(\la)$ belongs to the hyperplane $H_{n-1/2}$. In terms of the alcove geometry, one can think of it as the hyperplane of the alcove $\mathfrak{a}_n$ which is further away from the origin than $\la$. Note that since the $r$-restriction functor is exact, we have that $r-\mathrm{res}_{d-1}^{d}(C_{\bullet}(\la))$ is a complex. Consider a bipartition $\mu\in\Bip(d)$ such that $\mu$ is less dominant that $\la$ and $\mathsf{E}_{r}(\mu)$ is a bipartition. In the case we examine, all the bipartitions of $d$ with the aforementioned property come into pairs $(\nu^{+},\nu^{-})$ with $\nu^{+}\triangleleft \nu^{-}$ and $|\ell(\nu^{+})| = |\ell(\nu^{-})| + 1$. They also have the additional property 
\[
\mathsf{E}_{r}(\nu^{+}) = \mathsf{E}_{r}(\nu^{-}) = \nu
\]
where $\nu$ is linked with $\mathsf{E}_{r}(\la)$. Then we have that 
\[
r-\mathrm{res}_{d-1}^{d}(\Delta_{d}(\nu^{+})) = r-\mathrm{res}_{d-1}^{d}(\Delta_{d}(\nu^{-})) = \Delta_{d-1}(\nu).
\]
Now consider the homomorphism $\varphi_{\nu^{+}}^{\nu^{-}}\in\Hom_{\blob}(\Delta_{d}(\nu^{+}),\Delta_{d}(\nu^{-}))$ for some bipartition $\nu\triangleleft\mathsf{E}_{r}(\la)$. Under the $r$-restriction functor we have that 
\[
r-\mathrm{res}_{d-1}^{d}(\varphi_{\nu^{+}}^{\nu^{-}}) = 1_{\nu}\in\mathrm{End}_{\blob}(\Delta_{d-1}(\nu)).
\] 
In other words the identity morphism appears into all the differentials of the $r$-restricted complex $r-\mathrm{res}_{d-1}^{d}(C_{\bullet}(\la))$, hence the complex is exact. In particular the homology 
\[
H_{i}(r-\mathrm{res}_{d-1}^{d}(C_{\bullet}(\la))) = 0
\]
for any $i\ge 0$.
\item Suppose that the bipartition $\mathsf{E}_{r}(\la)$ belongs to the hyperplane $H_{n+1/2}$, that is the hyperplane closer to the origin. Recall that we denote by $\nu_{1}^{\prime}\in\Bip(d)$ the bipartition such that $|\ell(\nu_{1}^{\prime})| = |\ell(\la)| + 1$ with $\nu_{1}^{\prime}$ belonging to the positive alcoves. The pair of bipartitions $\mathsf{E}_{r}(\la)$, $\mathsf{E}_{r}(\nu_{1})\in\Bip(d-1)$ are of the form of Proposition \ref{BGG resolution for the wall case}. Apart form those bipartitions, all the rest bipartitions $\mu\in\Bip(d)$ which are strictly less dominant than $\la$ and $\mathsf{E}_{r}(\mu)\in\Bip(d)$, pair up in the exact same way as in the previous case when restricted under the $r$-restriction functor. Hence
\[
H_{i}(r-\mathrm{res}_{d-1}^{d}(C_{\bullet}(\la))) = 0
\]  
for $i > 0$. From Proposition \ref{BGG resolution for the wall case} we have that 
\[
H_{0}(r-\mathrm{res}_{d-1}^{d}(C_{\bullet}(\la))) = L(\la).
\] 
\item Suppose that the bipartition $\mathsf{E}_{r}(\la)$ remains to the alcove $\mathfrak{a}_n$. Then the complex $r-\mathrm{res}_{d-1}^{d}(C_{\bullet}(\la))$ is given by
\[
r-\mathrm{res}_{d-1}^{d}\left(\bigoplus_{\nu\trianglelefteq\la}\Delta_{d}(\nu)\langle\ell(\nu)\rangle\right)
\] 
with differentials given by 
\[
r-\mathrm{res}_{d-1}^{d}(\delta_{i})\colon r-\mathrm{res}_{d-1}^{d}(C_{i+1}(\la))\lxr r-\mathrm{res}_{d-1}^{d}(C_{i}(\la)).
\]
Note that if $\mathsf{Rem}_{r}(\nu)\neq \emptyset$, we have that 
\[
r-\mathrm{res}_{d-1}^{d}(\Delta_{d}(\nu)\langle\ell(\nu)\rangle) = \Delta_{d-1}(\mathsf{E}_{r}(\nu))\langle\ell(\nu) \rangle
\]
since $\ell(\nu) = \ell(\mathsf{E}_{r}(\nu))$, otherwise we have that 
\[
r-\mathrm{res}_{d-1}^{d}(\Delta_{d}(\nu)) = 0.
\]
Let $\nu,\nu^{\prime}\in\Bip(d)$ be bipartitions such that $\mathsf{Rem}_{r}(\nu), \mathsf{Rem}_{r}(\nu^{\prime})\neq \emptyset$. Then
\[
r-\mathrm{res}_{d-1}^{d}(\varphi_{\nu}^{\nu^{\prime}}) = \varphi_{\mathsf{E}_{r}(\nu)}^{\mathsf{E}_{r}(\nu^{\prime})}
\]
Hence we get that 
\[
r-\mathrm{res}_{d-1}^{d}(C_{\bullet}(\la)) = C_{\bullet}(\mathsf{E}_{r}(\lambda))
\]
and by the induction hypothesis we have that $H_{0}(C_{\bullet}(\mathsf{E}_{r}(\la))) = L_{d-1}(\mathsf{E}_{r}(\la))$, while $H_{i}(C_{\bullet}(\mathsf{E}_{r}(\la))) = 0$, for all $i >0$. Thus $r-\mathrm{res}_{d-1}^{d}(H_{0}(C_{\bullet}(\la))) = H_{0}(C_{\bullet}(\mathsf{E}_{r}(\la))) = L_{d-1}(\mathsf{E}_{r}(\la))$ and $r-\mathrm{res}_{d-1}^{d}(H_{i}(C_{\bullet}(\la))) = 0$, for all $i > 0$.
\end{itemize} 
Using the work we have done above we have proven that
\[
\mathrm{res}_{d-1}^{d}(H_{i}(C_{\bullet}(\la))) = 
\begin{cases}
\bigoplus_{r\in I}L_{d-1}(\mathsf{E}_{r}(\la)), \ \text{if} \ i = 0 \\
0, \hspace{2.95cm} \text{otherwise}.
\end{cases}
\]  
Moreover we have that the cokernel of the differential $\delta_1$ projects onto the simple representation $L(\la)$. The above argument gives us 
\[
\mathrm{res}_{d-1}^{d}(L_{d}(\la))\subset \bigoplus_{r\in I}L_{d-1}(\mathsf{E}_{r}(\la)).
\] 
In addition, by Theorem \ref{bases of simple modules}, we have that the cardinality of the basis of the simple representation $L_{d}(\la)$ is equal to the sum of the cardinalities of the bases for the simple representations $L_{d-1}(\mathsf{E}_{r}(\la))$, for all $r\in I$. Thus
\[
\mathrm{res}_{d-1}^{d}(L_{d}(\la)) = \bigoplus_{r\in I}L_{d-1}(\mathsf{E}_{r}(\la)).
\]
and we conclude that 
\[
\mathrm{res}_{d-1}^{d}(H_{i}(C_{\bullet}(\la))) = 
\begin{cases}
\mathrm{res}_{d-1}^{d}(L_{d}(\la)), \ \text{if} \ i = 0 \\
0, \hspace{2.15cm} \text{otherwise}
\end{cases}.
\]
Since $\mathrm{res}_{d-1}^{d}(L_{d}(\nu))\neq 0$, for any $\nu\trianglelefteq \la$, despite the fact that $r-\mathrm{res}_{d-1}^{d}(L_{d}(\nu)) = 0$, for some $r\in I$, we have that
\[
H_{i}(C_{\bullet}(\la)) = 
\begin{cases}
L_{d}(\la), \ \text{if} \ i = 0 \\
0, \hspace{0.8cm} \text{otherwise}.
\end{cases}
\]
and the proof is complete. \qedhere
\end{proof}


\end{document}